\documentclass{amsart}
\usepackage{amsmath, amsfonts, amsthm, framed}
\usepackage{xspace}
\usepackage{newunicodechar}
\usepackage{fullpage}
\usepackage{enumitem}
\usepackage{quiver}
\usepackage{fontawesome}
\usepackage[numbers]{natbib}
\usepackage{soul}
\usepackage{cancel}
\usepackage{xspace}
\usepackage{eucal}
\usepackage[colorlinks=true]{hyperref}
\usepackage{listings}

\definecolor{keywordcolor}{rgb}{0.7, 0.1, 0.1}   
\definecolor{commentcolor}{rgb}{0.4, 0.4, 0.4}   
\definecolor{symbolcolor}{rgb}{0, 0, 0.8}    
\definecolor{tacticcolor}{rgb}{0, 0, 0.8}    
\definecolor{sortcolor}{rgb}{0.1, 0.5, 0.1}      
\newcommand*{\lean}[1]{\lstinline{#1}\xspace} 
\theoremstyle{plain}
\newtheorem{theorem}{Theorem}[section]
\newcounter{maintheorems}
\newtheorem{maintheorem}{Theorem}[maintheorems]

\newcounter{maintheoremsbody}
\newtheorem{maintheorembody}{Theorem}[maintheoremsbody]

\newtheorem{lemma}[theorem]{Lemma}
\newtheorem{proposition}[theorem]{Proposition}
\newtheorem{corollary}[theorem]{Corollary}
\theoremstyle{definition}
\newtheorem{definition}[theorem]{Definition}
\newtheorem{remark}[theorem]{Remark}
\newtheorem{notation}[theorem]{Notation}
\newtheorem{construction}[theorem]{Construction}

\renewcommand{\epsilon}{\varepsilon}
\newcommand*{\N}{\mathbb{N}}
\newcommand*{\Z}{\mathbb{Z}}
\newcommand*{\A}{\mathcal{A}}
\newcommand*{\B}{\mathcal{B}}
\newcommand*{\C}{\mathcal{C}}
\newcommand*{\D}{\mathcal{D}}
\newcommand*{\E}{\mathcal{E}}
\newcommand*{\F}{\mathcal{F}}

\newcommand*{\dq}{\mathsf{dq}}
\newcommand*{\Hom}{\operatorname{Hom}}
\newcommand*{\op}{\operatorname{op}}
\newcommand*{\Lan}{\operatorname{Lan}}
\newcommand*{\colim}{\operatorname{colim}}

\newcommand*{\Set}{\mathsf{Set}}
\newcommand*{\FinSet}{\mathsf{FinSet}}
\newcommand*{\Sh}{\mathsf{Sh}}
\newcommand*{\PSh}{\mathsf{PSh}}
\newcommand*{\Top}{\mathsf{Top}}
\newcommand*{\CompHaus}{\mathsf{CompHaus}}
\newcommand*{\Profinite}{\mathsf{Profinite}}

\newcommand*{\CondSet}{\mathsf{CondSet}}
\newcommand*{\Mod}{\mathsf{Mod}}
\newcommand*{\CondMod}{\mathsf{CondMod}}
\newcommand*{\Cont}{\mathsf{Cont}}
\newcommand*{\LocConst}{\mathsf{LocConst}}
\newcommand*{\solid}{{\rule[0.36pt]{3.5pt}{3.5pt}}} 
\makeatletter
\newenvironment*{g@n@riclist}[1]{%
    \begin{enumerate}[label=#1]
    }{%
    \end{enumerate}
    }
\newenvironment{list_conditions}{
    \begin{g@n@riclist}{\arabic*\textup{)}}
    }{%
    \end{g@n@riclist}
    }

\makeatother

\newcommand*{\mathlib}{\textsc{mathlib}\xspace} 
\newcommand*{\Lean}{\textsc{Lean}\xspace} 

\usepackage{stackengine}
\stackMath
\setstackgap{S}{-2pt}
\newcommand{\func}{\Shortstack{\rightharpoonup \\ \hspace{-2.5pt}\rightharpoondown}}
\newunicodechar{⥤}{\ensuremath{\func}}

\title{A formal characterization of discrete condensed objects}
\author{Dagur Asgeirsson}
\address{University of Copenhagen}
\email{dagur@math.ku.dk}
\begin{document}

\begin{abstract}
    Condensed mathematics, developed by Clausen and Scholze over the last few years, proposes a generalization of topology with better categorical properties. 
    It replaces the concept of a topological space by that of a condensed set, which can be defined as a sheaf on a certain site of compact Hausdorff spaces. 
    Since condensed sets are supposed to be a generalization of topological spaces, one would like to be able to study the notion of discreteness.
    There are various ways to define what it means for a condensed set to be discrete.
    In this paper we describe them, and prove that they are equivalent.
    The results have been fully formalized in the \Lean proof assistant. 
\end{abstract}

\maketitle

\section{Introduction}
Condensed mathematics (originally defined in \cite{condensed}, see also \cite{msc-dagur,cat-found-cond, pyknotic,complex,analytic-stacks,analytic}) is a new framework which is suitable for applying algebraic techniques, such as homological algebra, in a setting where the objects of study are of a topological nature. In this framework, topological spaces are replaced by so-called \emph{condensed sets}. The goal of this paper is to explore one aspect of the connection between condensed sets and topological spaces --- the important example of discrete spaces. In condensed mathematics, the notion of discreteness becomes surprisingly subtle. 

All results presented in this paper, except for an application described in \S\ref{subsec:application}, have been formalized in the \Lean theorem prover \cite{lean4}, and integrated into its mathematical library \mathlib \cite{mathlib}. There are subtle set-theoretic issues that arise in the foundations of condensed mathematics, to which the type theory of \Lean provides a satisfactory solution, as explained in \S\ref{sec:size} (a short overview of the design principles of \Lean and \mathlib, as relevant to condensed mathematics, is given in \cite[\S2]{cat-found-cond}). Another approach to these issues is to avoid the problem by switching to the theory of \emph{light condensed objects} (see the recent lecture series \cite{analytic-stacks}). I have also formalized the foundations of light condensed mathematics. This material is in \mathlib, and the formalization there provides a possibility to unify the two as described in \S\ref{sec:size}. The material in \S\ref{sec:functors} was formalized in such a unified way (see \S\ref{subsec:generality} for more details). 

One of the main contributions of the present work is that it describes and completely resolves a big subtlety in the theory of discrete condensed sets, which was discovered during the formalization process. Discrete condensed sets are most naturally defined as constant sheaves. It turns out to be quite difficult to prove that a constant sheaf on the defining site of condensed sets (and light condensed sets) is in fact given by locally constant maps --- the analogue of a result that is well known in the case of sheaves on a topological space. This result is important because without it, there is no chance of an explicit description of the sections of a discrete condensed set. More broadly, it is my hope that this paper further improves the state of the literature on condensed mathematics, which prior to \cite{cat-found-cond} consisted mostly of online lecture notes \cite{condensed,analytic,complex} and videos \cite{analytic-stacks}.

We provide several equivalent conditions on a condensed set that characterize it as discrete (Theorem~\ref{thm:A}, Theorem~\ref{thm:B}, Theorem~\ref{thm:summary}). Moreover, we prove that the characterization of discrete condensed sets carries over nicely to condensed modules over a ring (Theorem~\ref{thm:C}, Theorem~\ref{thm:summary_mod}). All the results hold in both the original setting of condensed objects, and that of light condensed objects.

Throughout the text, we use the symbol \faExternalLink~for external links. Every mathematical declaration will be accompanied by such a link to the corresponding formal statement in \mathlib. All the links are to the same commit to the master branch, ensuring that the links stay usable. 

\subsection{Main results} 
We give the definition of a condensed set and some related terminology, before stating the main results of this paper. 

Recall that a \emph{condensed set} \href{https://github.com/leanprover-community/mathlib4/blob/ffb2f0dbf80ad18fa75f85bd0facee6737b59acf/Mathlib/Condensed/Basic.lean#L52}{\faExternalLink} is a sheaf for the coherent topology on the category of compact Hausdorff spaces\footnote{This is the \mathlib definition. The original definition uses the category of profinite sets as the defining site. These two definitions are equivalent, as formalized in \mathlib and described in \cite{cat-found-cond}.}. More concretely, it is a presheaf $X : \CompHaus^{\op} \to \Set$, satisfying the two properties \href{https://github.com/leanprover-community/mathlib4/blob/ffb2f0dbf80ad18fa75f85bd0facee6737b59acf/Mathlib/Condensed/Explicit.lean#L164-L166}{\faExternalLink}:
\begin{list_conditions}
    \item $X$ preserves finite products: in other words, for every finite family $(T_i)$ of compact Hausdorff spaces, the natural map 
    \[
        X\Bigl(\coprod_i T_i\Bigr) \longrightarrow \prod_i X(T_i)    
    \]
    is a bijection.
    \item For every continuous surjection $\pi \colon S \to T$ of compact Hausdorff spaces, the diagram 
    \[\begin{tikzcd}
        {X(T)} & {X(S)} & {X(S \times_T S)}
        \arrow["{X(\pi)}", from=1-1, to=1-2]
        \arrow[shift left, from=1-2, to=1-3]
        \arrow[shift right, from=1-2, to=1-3]
    \end{tikzcd}\]
    is an equalizer (the two parallel morphisms being induced by the projections in the pullback). 
\end{list_conditions}
Furthermore, there is a functor from topological spaces to condensed sets \href{https://github.com/leanprover-community/mathlib4/blob/ffb2f0dbf80ad18fa75f85bd0facee6737b59acf/Mathlib/Condensed/TopComparison.lean#L141-L142}{\faExternalLink}, which takes a topological space $X$ to the sheaf $\Cont(-,X)$, which maps a compact Hausdorff space $S$ to the set of continuous maps $S \to X$. This functor is fully faithful when restricted to \emph{compactly generated spaces} \href{https://github.com/leanprover-community/mathlib4/blob/ffb2f0dbf80ad18fa75f85bd0facee6737b59acf/Mathlib/Condensed/TopCatAdjunction.lean#L193-L195}{\faExternalLink} \href{https://github.com/leanprover-community/mathlib4/blob/ffb2f0dbf80ad18fa75f85bd0facee6737b59acf/Mathlib/Topology/Category/CompactlyGenerated.lean#L31-L35}{\faExternalLink} (these include all reasonable topological spaces most mathematicians care about). Continuous maps from a point to a topological space $X$ identify with the underlying set $X$, which is why we define the \emph{underlying set} \href{https://github.com/leanprover-community/mathlib4/blob/ffb2f0dbf80ad18fa75f85bd0facee6737b59acf/Mathlib/Condensed/Discrete/Basic.lean#L42-L44}{\faExternalLink} of a condensed set $X$ to be the set $X(*)$. Even though condensed sets do not form a concrete category, we will sometimes call the functor $X \mapsto X(*)$ the \emph{forgetful functor}. By analogy with topological spaces, we would like to be able to study \emph{discrete condensed sets}. But what does it mean for a condensed set to be discrete?

Let's take a look at discreteness in topology. A discrete topological space $X$ is one in which all subsets are open. It is characterized by the property that every map $X \to Y$, where $Y$ is any topological space, is continuous. In other words, we have an adjoint pair of functors
\[\begin{tikzcd}
	\Set & \Top
	\arrow[""{name=0, anchor=center, inner sep=0}, "\delta", curve={height=-6pt}, from=1-1, to=1-2]
	\arrow[""{name=1, anchor=center, inner sep=0}, "U", curve={height=-6pt}, from=1-2, to=1-1]
	\arrow["\dashv"{anchor=center, rotate=-90}, draw=none, from=0, to=1]
\end{tikzcd}\]
where $U$ denotes the forgetful functor and $\delta$ the functor which equips a set with the discrete topology.
This suggests that we should try to define a functor $\Set \to \CondSet$ which ``equips a set with a discrete condensed structure'', which is left adjoint to the underlying set functor $\CondSet \to \Set$.

An obvious candidate for this functor $\Set \to \CondSet$ is \href{https://github.com/leanprover-community/mathlib4/blob/ffb2f0dbf80ad18fa75f85bd0facee6737b59acf/Mathlib/Condensed/Discrete/Basic.lean#L35-L36}{\faExternalLink}
\begin{align*}
    \underline{(-)} : \Set & \to \CondSet
\end{align*}
where $\underline{X}$ denotes the condensed set given by the constant sheaf at $X$. It is almost by definition left adjoint to the underlying set functor $\CondSet \to \Set$ \href{https://github.com/leanprover-community/mathlib4/blob/ffb2f0dbf80ad18fa75f85bd0facee6737b59acf/Mathlib/Condensed/Discrete/Basic.lean#L50-L51}{\faExternalLink}. 

Another candidate comes from the functor from topological spaces to condensed sets that we already have, which takes a topological space $X$ to the sheaf of continuous maps to $X$. Precomposing this functor with the functor which equips a set $X$ with the discrete topology, we obtain a functor $\Set \to \CondSet$. Rephrasing the above, this is the functor which takes a set $X$ to the sheaf of locally constant maps to $X$ \href{https://github.com/leanprover-community/mathlib4/blob/ffb2f0dbf80ad18fa75f85bd0facee6737b59acf/Mathlib/Condensed/Discrete/LocallyConstant.lean#L228-L237}{\faExternalLink}. We denote this functor by 
\begin{align*}
    L : \Set & \to \CondSet \\
    X & \mapsto \LocConst(-, X)
\end{align*}

Each of these functors has its advantages and disadvantages. The constant sheaf functor has nice abstract properties and is obviously a left adjoint to the underlying set functor, while the functor $L$ has better definitional properties given by its concrete description.
Our first objective is to construct a natural isomorphism $L \cong \underline{(-)}$. We do this by constructing an adjunction $L \dashv U$ where $U : \CondSet \to \Set$ is the underlying set functor mapping a condensed set $X$ to $X(*)$. 

We can now give the official definition of a discrete condensed set.
\begin{definition}\label{def:discrete_condensed_set} \href{https://github.com/leanprover-community/mathlib4/blob/ffb2f0dbf80ad18fa75f85bd0facee6737b59acf/Mathlib/Condensed/Discrete/Characterization.lean#L45}{\faExternalLink} \href{https://github.com/leanprover-community/mathlib4/blob/ffb2f0dbf80ad18fa75f85bd0facee6737b59acf/Mathlib/CategoryTheory/Sites/ConstantSheaf.lean#L80-L81}{\faExternalLink}
    A condensed set $X$ is \emph{discrete} if it is in the essential image of the functor 
    \[
        \underline{(-)} : \Set \to \CondSet,
    \] 
    i.e. if there exists some set $X'$ and an isomorphism 
    \[ 
        X \cong \underline{X'}.
    \]
\end{definition}

The discussion preceding Definition~\ref{def:discrete_condensed_set} immediately gives our first main theorem:
\begin{maintheorem}\label{thm:A} \href{https://github.com/leanprover-community/mathlib4/blob/ffb2f0dbf80ad18fa75f85bd0facee6737b59acf/Mathlib/Condensed/Discrete/Characterization.lean#L74-L104}{\faExternalLink}
    A condensed set is discrete if and only if it is in the essential image of the functor 
    \[
        L : \Set \to \CondSet,
    \] 
    i.e. if there exists some set $X'$ and an isomorphism of condensed sets
    \[
        X \cong \LocConst(-,X')
    \]
\end{maintheorem}

A condensed set is completely determined by its values on profinite sets (see \cite[Theorem 7.7]{cat-found-cond}). Profinite sets are those topological spaces which can be written as a cofiltered limit of finite discrete sets. A particular limit formula for a profinite set $S$ is given by writing it as the limit of its discrete quotients, described in more detail in \S\ref{sec:colimit}. When we write $S = \varprojlim_i S_i$, we mean this formula. We can now state the second of our main theorems:

\begin{maintheorem}\label{thm:B} \href{https://github.com/leanprover-community/mathlib4/blob/ffb2f0dbf80ad18fa75f85bd0facee6737b59acf/Mathlib/Condensed/Discrete/Characterization.lean#L74-L104}{\faExternalLink}
    A condensed set $X$ is discrete if and only if for every profinite set $S = \varprojlim_i S_i$, the canonical map $X(S) \to \varinjlim_i X(S_i)$ is an isomorphism.
\end{maintheorem}

The definition of a discrete condensed set can be extended to condensed $R$-modules in an obvious way, leading to the third and final main theorem:

\begin{maintheorem}\label{thm:C} \href{https://github.com/leanprover-community/mathlib4/blob/ffb2f0dbf80ad18fa75f85bd0facee6737b59acf/Mathlib/Condensed/Discrete/Characterization.lean#L112-L115}{\faExternalLink}
    Let $R$ be a ring. A condensed $R$-module is discrete (i.e. constant as a sheaf) if and only if its underlying condensed set is discrete.
\end{maintheorem}

\subsection{Application}\label{subsec:application}
The colimit characterization of discrete condensed sets given by Theorem~\ref{thm:B} (or more precisely, its analogue for condensed abelian groups, which follows from Theorems~\ref{thm:B} and~\ref{thm:C} together) is used when setting up the theory of \emph{solid abelian groups}. The results presented in this subsection have not yet been formalized, but a formalization is well within reach and is currently only blocked by the fact that the closed symmetric monoidal structure on sheaf categories has not been completely formalized yet (this is work in progress by the author and Joël Riou).

Solid abelian groups are condensed analogues of complete topological groups. We say that a condensed abelian group $A$ is \emph{solid} \href{https://github.com/leanprover-community/mathlib4/blob/ffb2f0dbf80ad18fa75f85bd0facee6737b59acf/Mathlib/Condensed/Solid.lean#L72-L74}{\faExternalLink} if for every profinite set $T = \varprojlim_i T_i$, the natural map
\[
    \Z[T] \to \varprojlim_i Z[T_i]
\]
induces a bijection
\[
    \Hom\left(\varprojlim_i \Z[T_i], A\right) \to \Hom\left(\Z[T], A\right).
\]
We denote by $\Z[T]^\solid$ the condensed abelian group $\varprojlim_i Z[T_i]$. These are compact projective generators of the category of solid abelian groups. To get the theory off the ground, one needs to prove that this category has some nice properties (described by \cite[Theorem 5.8]{condensed}), which relies on a structural result \cite[Corollary 5.5]{condensed} about these building blocks. That result says that for every profinite set $T$, there is a set $I$ and an isomorphism 
\[
    \Z[T]^\solid \cong \prod_{i \in I} \Z.
\]
To prove this, one provides an isomorphism 
\[
    \Z[T]^\solid \cong \underline{\Hom}\left(\underline{\Hom}\left(\Z[T], \Z\right), \Z\right)
\]
and then proves that the internal hom $\underline{\Hom}\left(\Z[T], \Z\right)$ is a discrete condensed abelian group, using the colimit characterization\footnote{This isomorphism, and the proof that the internal hom is discrete, is explained in an unpublished note by the author, see \cite[2.4-2.7]{solid-spectra}}. This means that it is simply given by the abelian group of continuous maps $C(T, \Z)$, and thus Nöbeling's theorem \cite[Theorem 5.4]{condensed} \href{https://github.com/leanprover-community/mathlib4/blob/ffb2f0dbf80ad18fa75f85bd0facee6737b59acf/Mathlib/Topology/Category/Profinite/Nobeling.lean#L1834-L1838}{\faExternalLink}, which was formalized by the author in \cite{nobeling}, applies. Nöbeling's theorem gives an isomorphism
\[
    \underline{\Hom}\left(\Z[T], \Z\right) \cong \bigoplus_{i\in I} \Z,
\]
and the desired isomorphism 
\[
    \Z[T]^\solid \cong \underline{\Hom}\left(\bigoplus_{i\in I} \Z, \Z\right) \cong \prod_{i \in I} \Z
\]
follows.

\subsection{Terminology and notation}
Throughout the paper, we use the same terminology as \mathlib. In particular, for an overview of the terminology surrounding Grothendieck topologies and sheaf theory that \mathlib (and this paper) uses, we refer to \cite[\S 3.1]{cat-found-cond}. We will need the notion of initial and final functors. Sometimes the latter is referred to as \emph{cofinal}, which is unfortunate because one might think that \emph{cofinal} refers to the dual of \emph{final}. We explain precisely what we mean by these in \S\ref{subsec:initial-final}.

\subsection{Outline}
We prove Theorem~\ref{thm:A} in \S\ref{sec:functors}, Theorem~\ref{thm:B} in \S\ref{sec:colimit}, and Theorem~\ref{thm:C} in \S\ref{sec:modules}. The general categorical setup to prove Theorem~\ref{thm:C} is in \S\ref{sec:constant}. We gather a few generalities in category theory which are perhaps not completely standard in \S\ref{sec:adjunctions}. We discuss some set theoretic issues which are important to consider when setting up the formalized theory of condensed mathematics in \S\ref{sec:size}. \textsc{Mathlib} has an unusual approach to sheafification; this is described in \S\ref{sec:sheafification}. Finally, we tie together the main results as Theorems~\ref{thm:summary} and~\ref{thm:summary_mod} in \S\ref{sec:summary}, giving a characterization of discrete condensed sets and discrete condensed modules over any ring.

\section{Preliminaries}
\subsection{Generalities in category theory}\label{sec:adjunctions}
\subsubsection{Adjunctions}
We prove two useful results about adjunctions. The results in this subsection were formalized by the author as part of the prerequisites for this project. 
Let $\C, \D$ be categories and 
\[\begin{tikzcd}
  \C & \D
  \arrow[""{name=0, anchor=center, inner sep=0}, "L", curve={height=-6pt}, from=1-1, to=1-2]
  \arrow[""{name=1, anchor=center, inner sep=0}, "R", curve={height=-6pt}, from=1-2, to=1-1]
  \arrow["\dashv"{anchor=center, rotate=-90}, draw=none, from=0, to=1]
\end{tikzcd}\]
a pair of adjoint functors.

\begin{proposition}\label{prop:adjunction-generality-1} \href{https://github.com/leanprover-community/mathlib4/blob/ffb2f0dbf80ad18fa75f85bd0facee6737b59acf/Mathlib/CategoryTheory/Adjunction/FullyFaithful.lean#L187-L194}{\faExternalLink}
  Suppose that $L$ is fully faithful and let $X$ be an object of $\D$. Then $X$ is in the essential image of $L$ if and only if the counit induces an isomorphism $L(R(X)) \to X$. 
\end{proposition}
\begin{proof}
  The backward direction is obvious. For the forward direction, pick an isomorphism of the form $X \cong L(Y)$. Then the unit gives a one-sided inverse to the counit $L(R(L(Y))) \to L(Y)$ by one of the triangle identities. Since the left adjoint is fully faithful, the unit is an isomorphism and hence a two-sided inverse. We conclude that the counit induces an isomorphism $L(R(L(Y))) \to L(Y)$, and since the counit is a natural transformation, and there is an isomorphism $X \cong L(Y)$, this means it also induces an isomorphism $L(R(X)) \to X$.
\end{proof}

\begin{proposition}\label{prop:adjunction-generality-2} \href{https://github.com/leanprover-community/mathlib4/blob/ffb2f0dbf80ad18fa75f85bd0facee6737b59acf/Mathlib/CategoryTheory/Monad/Adjunction.lean#L103-L104}{\faExternalLink}
  If there exists a natural isomorphism $R \circ L \cong \operatorname{Id}$, then the unit $\operatorname{Id} \to R \circ L$ is an isomorphism.
\end{proposition}
\begin{proof}
  Let $i : R \circ L \to \operatorname{Id}$ be an isomorphism. Let $(\operatorname{Id}, \eta, \mu)$ denote the monad obtained by transporting \href{https://github.com/leanprover-community/mathlib4/blob/ffb2f0dbf80ad18fa75f85bd0facee6737b59acf/Mathlib/CategoryTheory/Monad/Basic.lean#L301-L326}{\faExternalLink} the monad on $R \circ L$ induced by the adjunction, along the isomorphism $i$. The inverse of the unit of the adjunction is then given by $\mu \circ i$. The fact that this is an inverse follows from the coherence conditions of the monad and the fact that any monad on the identity functor is commutative. 
\end{proof}

For a more detailed proof of Proposition~\ref{prop:adjunction-generality-2}, we refer to \mathlib. This same proof idea is described in \cite[Lemma A.1.1.1]{elephant}

\begin{corollary}\label{cor:adjunction-generality} \href{https://github.com/leanprover-community/mathlib4/blob/ffb2f0dbf80ad18fa75f85bd0facee6737b59acf/Mathlib/CategoryTheory/Monad/Adjunction.lean#L110-L112}{\faExternalLink}
  If there exists a natural isomorphism $R \circ L \cong \operatorname{Id}$, then $L$ is fully faithful. 
\end{corollary}

\subsubsection{Initial and final functors}\label{subsec:initial-final}
The results in this subsection were not formalized by the author and were already in \mathlib before the start of this project. The purpose of stating them here is to clarify the terminology. We omit all proofs and refer to the \Lean code for details.

\begin{definition}\label{def:initial-final} \href{https://github.com/leanprover-community/mathlib4/blob/ffb2f0dbf80ad18fa75f85bd0facee6737b59acf/Mathlib/CategoryTheory/Limits/Final.lean#L96-L97}{\faExternalLink} \href{https://github.com/leanprover-community/mathlib4/blob/ffb2f0dbf80ad18fa75f85bd0facee6737b59acf/Mathlib/CategoryTheory/Limits/Final.lean#L88-L89}{\faExternalLink}
  Let $\C$ and $\D$ be categories, and let $F : \C \to \D$ be a functor. 
  \begin{itemize}
    \item $F$ is \emph{initial} if the comma category $F/\D$ (whose objects are morphisms of the form $F(X) \to Y$ with $X \in \C$ and $Y \in \D$) is connected. 
    \item $F$ is \emph{final} if the comma category $\D/F$ (whose objects are morphisms of the form $Y \to F(X)$ with $X \in \C$ and $Y \in \D$) is connected. 
  \end{itemize}
\end{definition}

\begin{proposition}\label{prop:initial-final-op} \href{https://github.com/leanprover-community/mathlib4/blob/ffb2f0dbf80ad18fa75f85bd0facee6737b59acf/Mathlib/CategoryTheory/Limits/Final.lean#L101-L117}{\faExternalLink}
  A functor is initial if and only if its opposite is final, and vice versa. 
\end{proposition}

\begin{proposition}\label{prop:initial-final-limit} \href{https://github.com/leanprover-community/mathlib4/blob/ffb2f0dbf80ad18fa75f85bd0facee6737b59acf/Mathlib/CategoryTheory/Limits/Final.lean#L585-L592}{\faExternalLink} \href{https://github.com/leanprover-community/mathlib4/blob/ffb2f0dbf80ad18fa75f85bd0facee6737b59acf/Mathlib/CategoryTheory/Limits/Final.lean#L619-L620}{\faExternalLink} \href{https://github.com/leanprover-community/mathlib4/blob/ffb2f0dbf80ad18fa75f85bd0facee6737b59acf/Mathlib/CategoryTheory/Limits/Final.lean#L302-L309}{\faExternalLink} \href{https://github.com/leanprover-community/mathlib4/blob/ffb2f0dbf80ad18fa75f85bd0facee6737b59acf/Mathlib/CategoryTheory/Limits/Final.lean#L336-L337}{\faExternalLink}
  Let $\E$ be a category and $G : \D \to \E$ be a functor. Then
  \begin{itemize}
    \item If $F$ is initial, then $G$ has a limit if and only if $G \circ F$ has a limit. In this case, the canonical map 
    \[\lim G \to \lim (G \circ F)\]
    is an isomorphism.
    \item If $F$ is final, then $G$ has a colimit if and only if $G \circ F$ has a colimit. In this case, the canonical map 
    \[\colim (G \circ F) \to \colim G \]
    is an isomorphism.
  \end{itemize}
\end{proposition}

\subsection{Set-theoretic issues in condensed mathematics}\label{sec:size}
In sheaf theory, one sometimes needs to be careful about the size of categories involved. For example, the category of sheaves of sets on a large site cannot strictly speaking have sheafification. This is because the process of sheafification involves taking a colimit which is the same size as the defining site. 

This issue arises in condensed mathematics, and can be solved in various ways. Originally it was done in \cite{condensed} by altering the definition of a condensed set slightly. Instead of defining it as a sheaf on the whole site of compact Hausdorff spaces, one defines a $\kappa$-condensed set to be a sheaf on the essentially small site of $\kappa$-small compact Hausdorff spaces, where $\kappa$ is a certain type of cut-off cardinal. Then one takes a colimit of these categories of $\kappa$-condensed sets and defines a condensed set to be an object in this colimit in the category of categories. 

The way this issue is solved in \cite{pyknotic} is to make the sheaves take values in a larger category of sets. So a condensed set is defined as a sheaf of large sets on the site of all small compact Hausdorff spaces. This is the approach we take in \mathlib as well, by taking advantage of the universe hierarchy built in to the type theory of \Lean (for a short explanation of the universe hierarchy, see \cite[\S2.5]{cat-found-cond}; for more details, see \cite[Chapter 2]{tpil}). A \lean{u}-condensed set is defined as a sheaf on the site \lean{CompHaus.\{u\}} of all compact Hausdorff spaces in the universe \lean{u}, valued in the category \lean{Type (u+1)} of types in the universe \lean{u+1} \href{https://github.com/leanprover-community/mathlib4/blob/ffb2f0dbf80ad18fa75f85bd0facee6737b59acf/Mathlib/Condensed/Basic.lean#L42-L43}{\faExternalLink} \href{https://github.com/leanprover-community/mathlib4/blob/ffb2f0dbf80ad18fa75f85bd0facee6737b59acf/Mathlib/Condensed/Basic.lean#L52}{\faExternalLink}:

\begin{lstlisting}
def Condensed (C : Type w) [Category.{v} C] :=
  Sheaf (coherentTopology CompHaus.{u}) C  

abbrev CondensedSet := Condensed.{u} (Type (u+1))
\end{lstlisting}

The definition of a condensed object in a general category \lean{C} has three universe parameters \lean{u}, \lean{v}, and \lean{w} where the latter two are the size of the sets of morphisms and objects of \lean{C} respectively. This is a recurring theme in the category theory library in \mathlib. A \emph{small category} in this context is category whose objects and morphisms live in the same universe \lean{u}. More precisely, the type \lean{C} of objects is a term of \lean{Type u} and for each pair of objects \lean{X Y : C}, the type \lean{X ⟶ Y} of morphisms between them is also a term of \lean{Type u}. In a \emph{large category} the objects form a type of size \lean{u+1} and the morphisms between two objects a type of size \lean{u}.

Many theorems in category theory hold for categories regardless of whether they are small or large. Others require modifications to depending on whether some categories in the statement are small or large. \textsc{Mathlib} strives to avoid code duplication as much as possible, and therefore, theorems in category theory often have assumptions of the form 
\begin{lstlisting}
variable (C : Type u) [Category.{v} C]
\end{lstlisting}
This means that the objects and morphisms of \lean{C} live in completely arbitrary universes, which specializes directly to both small and large categories.

In an informal mathematics text, one might see a definition which would start as follows:

Let $(\C, J)$ be a small site and let $F : \C^{\op} \to \Set$ be a presheaf. The \emph{sheafification} $\hat{F}$ of $F$ is a sheaf of sets for $J$ on $\C$ defined as follows \dots. 

And it would then go on to describe the values of this sheaf as a colimit in $\Set$, which exists because $\C$ was assumed to be small. 

In \mathlib, the approach taken in this situation is instead to depend on
\begin{lstlisting}
variable {C : Type u} [Category.{v} C] (J : GrothendieckTopology C) 
    (F : Cᵒᵖ ⥤ Type (max u v))
\end{lstlisting}
This specializes directly to the case of a small defining site (\lean{[Category.\{u\} C]}), because in that case the target category of the presheaf becomes \lean{Type u}. But it also allows for a more general construction, namely sheafifying presheaves on a large site (\lean{\{C : Type (u+1)\} [Category.\{u\} C]}), by making sure they take values in a category of large enough sets (\lean{Type (u+1)}). This way, sheafification is possible in the setting of \lean{CondensedSet} in \mathlib.


The size issues in condensed mathematics described above can be avoided by modifying the theory slightly. Instead of the category of compact Hausdorff spaces or that of profinite sets, we can take the category of \emph{light profinite sets} \href{https://github.com/leanprover-community/mathlib4/blob/ffb2f0dbf80ad18fa75f85bd0facee6737b59acf/Mathlib/Topology/Category/LightProfinite/Basic.lean#L44-L45}{\faExternalLink} as our defining site. This category consists of those profinite sets that can be written as a sequential limit $\varprojlim_{n \in \N} S_n$ of finite discrete sets \href{https://github.com/leanprover-community/mathlib4/blob/ffb2f0dbf80ad18fa75f85bd0facee6737b59acf/Mathlib/Topology/Category/LightProfinite/Basic.lean#L245-L251}{\faExternalLink}. Equivalently \href{https://github.com/leanprover-community/mathlib4/blob/ffb2f0dbf80ad18fa75f85bd0facee6737b59acf/Mathlib/Topology/Category/LightProfinite/Basic.lean#L336-L349}{\faExternalLink}, it is the category of all second countable, totally disconnected compact Hausdorff spaces (this is the \mathlib definition). This category is essentially small \href{https://github.com/leanprover-community/mathlib4/blob/ffb2f0dbf80ad18fa75f85bd0facee6737b59acf/Mathlib/Topology/Category/LightProfinite/Basic.lean#L404-L406}{\faExternalLink}, so we can define a \emph{light condensed set} as a sheaf of (small) sets on the coherent site of light profinite sets. In \textsc{mathlib} \href{https://github.com/leanprover-community/mathlib4/blob/ffb2f0dbf80ad18fa75f85bd0facee6737b59acf/Mathlib/Condensed/Light/Basic.lean#L25-L26}{\faExternalLink} \href{https://github.com/leanprover-community/mathlib4/blob/ffb2f0dbf80ad18fa75f85bd0facee6737b59acf/Mathlib/Condensed/Light/Basic.lean#L35}{\faExternalLink}:
\begin{lstlisting}
def LightCondensed (C : Type w) [Category.{v} C] :=
  Sheaf (coherentTopology LightProfinite.{u}) C

abbrev LightCondSet := LightCondensed.{u} (Type u)
\end{lstlisting}
A common generalization of condensed and light condensed sets is a term of the type 
\begin{lstlisting}
Sheaf (coherentTopology (CompHausLike.{u} P)) (Type (max u w))
\end{lstlisting}
Here, \lean{P} is a predicate on topological spaces, \lean{CompHausLike.\{u\} P} denotes the category of compact Hausdorff spaces in the universe \lean{u}, satisfying \lean{P} \href{https://github.com/leanprover-community/mathlib4/blob/ffb2f0dbf80ad18fa75f85bd0facee6737b59acf/Mathlib/Topology/Category/CompHausLike/Basic.lean#L72-L80}{\faExternalLink}, and \lean{w} is an auxiliary universe variable. In \mathlib, the categories \lean{CompHaus} and \lean{LightProfinite} are defined as abbreviations for \lean{CompHausLike P} for the relevant predicates \lean{P} \href{https://github.com/leanprover-community/mathlib4/blob/ffb2f0dbf80ad18fa75f85bd0facee6737b59acf/Mathlib/Topology/Category/CompHaus/Basic.lean#L43}{\faExternalLink} \href{https://github.com/leanprover-community/mathlib4/blob/ffb2f0dbf80ad18fa75f85bd0facee6737b59acf/Mathlib/Topology/Category/LightProfinite/Basic.lean#L44-L45}{\faExternalLink}. This is a recent design by the author with the purpose of being able to prove results about categories such as \lean{CompHaus} and \lean{LightProfinite} simultaneously.

All constructions in this paper carry over to the setting of light condensed objects. In the formalization, this is done either in each case separately or by arguing directly about the common generalization described in the previous paragraph. However, for clarity of exposition, we focus in this paper on the case of condensed sets, and let the reader themselves either fill in the details for the translation to light condensed sets, or look at the \Lean code. 

\subsection{Sheafification in mathlib}\label{sec:sheafification}
\textsc{Mathlib} has a somewhat unusual approach to sheafification. In this section we define some terminology which may be non-standard in the literature, but follows \mathlib conventions. Many of these constructions were formalized in part by the author, but Construction~\ref{const:preservesSheafification_diagram} and the material related to preserving sheafification was mostly formalized by Joël Riou.

\begin{definition}\label{def:HasSheafify} \href{https://github.com/leanprover-community/mathlib4/blob/ffb2f0dbf80ad18fa75f85bd0facee6737b59acf/Mathlib/CategoryTheory/Sites/Sheafification.lean#L33}{\faExternalLink} \href{https://github.com/leanprover-community/mathlib4/blob/ffb2f0dbf80ad18fa75f85bd0facee6737b59acf/Mathlib/CategoryTheory/Sites/Sheafification.lean#L42-L44}{\faExternalLink}
    Let $\C$ be a category equipped with a Grothendieck topology $J$ and let $\A$ be a category. 
    \begin{itemize}
        \item We say that $\A$ \emph{has weak sheafification with respect to $J$} if the inclusion functor $\Sh_J(\C, \A) \to \PSh(\C, \A)$ has a left adjoint (called the \emph{sheafification functor}). 
        \item We say that $\A$ \emph{has sheafification with respect to $J$} if it has weak sheafification and the sheafification functor preserves finite limits.
    \end{itemize}
\end{definition}

Previously, \mathlib only contained definitions of sheafification for sheaves valued in concrete categories with a long list of assumptions about existence of certain limits and colimits as well as the forgetful functor preserving these. Moreover, it only worked for presheaves of large sets (or other large concrete categories satisfying the required properties) defined on a large site, or presheaves of small sets defined on a small site, not for presheaves of small sets defined on a large, essentially small site. The type classes \lean{HasWeakSheafify} and \lean{HasSheafify} were added to \mathlib by the author while adding the possibility to sheafify presheaves of small sets defined on a large, essentially small site. This streamlined many files which previously contained said long lists of assumptions, which now instead only contain an assumption of the form \lean{HasWeakSheafify} or \lean{HasSheafify}.

\begin{definition}\label{def:SheafCompose} \href{https://github.com/leanprover-community/mathlib4/blob/ffb2f0dbf80ad18fa75f85bd0facee6737b59acf/Mathlib/CategoryTheory/Sites/Whiskering.lean#L38-L40}{\faExternalLink} \href{https://github.com/leanprover-community/mathlib4/blob/ffb2f0dbf80ad18fa75f85bd0facee6737b59acf/Mathlib/CategoryTheory/Sites/Whiskering.lean#L45-L50}{\faExternalLink}
    Let $\C$ be a category equipped with a Grothendieck topology $J$, let $\A$ and $\B$ be categories and $U : \A \to \B$ a functor. We say that $U$ \emph{has sheaf composition with respect to $J$} if for every $J$-sheaf $F$ with values in $\A$, the presheaf $U \circ F$ is a sheaf. This yields a functor $\Sh_J(\C, \A) \to \Sh_J(\C, \B)$ which we call the \emph{sheaf composition functor associated to $U$}.
\end{definition}

\begin{proposition}\label{prop:reflectsIso_etc} \href{https://github.com/leanprover-community/mathlib4/blob/ffb2f0dbf80ad18fa75f85bd0facee6737b59acf/Mathlib/CategoryTheory/Sites/Whiskering.lean#L58-L69}{\faExternalLink}
    Let $\C, J, \A, \B$ and $U$ be as in Definition~\ref{def:SheafCompose}. Suppose that $U$ has sheaf composition with respect to $J$. Denote by $F_U$ the sheaf composition functor associated to $U$. Then the following holds:
    \begin{enumerate}
        \item If $U$ is faithful, then $F_U$ is faithful.
        \item If $U$ is full and faithful, then $F_U$ is full.
        \item If $U$ reflects isomorphisms, then $F_U$ reflects isomorphisms.
    \end{enumerate}
\end{proposition}

\begin{definition}\label{def:PreservesSheafification} \href{https://github.com/leanprover-community/mathlib4/blob/ffb2f0dbf80ad18fa75f85bd0facee6737b59acf/Mathlib/CategoryTheory/Sites/PreservesSheafification.lean#L57-L58}{\faExternalLink} \href{https://github.com/leanprover-community/mathlib4/blob/ffb2f0dbf80ad18fa75f85bd0facee6737b59acf/Mathlib/CategoryTheory/Sites/PreservesSheafification.lean#L69-L74}{\faExternalLink}
    Let $\C$ be a category equipped with a Grothendieck topology $J$, let $\A$ and $\B$ be categories and $U : \A \to \B$ a functor. Suppose that $A$ and $B$ have weak sheafification with respect to $J$. We say that $U$ \emph{preserves sheafification} if for every morphism $\phi : P_1 \to P_2$ of $\A$-valued presheaves which becomes an isomorphism after sheafification, the whiskering $U\phi$ also becomes an isomorphism after sheafification.\footnote{In \mathlib, this definition is stated slightly differently without assuming that the categories have weak sheafification. In our application, the categories do have sheafification and our definition is equivalent.}
\end{definition}

\begin{construction}\label{const:preservesSheafification_diagram} \href{https://github.com/leanprover-community/mathlib4/blob/ffb2f0dbf80ad18fa75f85bd0facee6737b59acf/Mathlib/CategoryTheory/Sites/PreservesSheafification.lean#L157-L167}{\faExternalLink}
    Let $\C, J, \A, \B$, and $U$ be like in Definition~\ref{def:PreservesSheafification}. Suppose that $U$ has sheaf composition with respect to $J$. Denote by $G_1$ and $G_2$ the sheafification functors associated of $\A$ and $\B$ respectively. Consider the following square of functors:
    \[\begin{tikzcd}
        {\PSh(\C,\A)} & {\PSh(\C, \B)} \\
        {\Sh_J(\C,\A)} & {\Sh_J(\C,\B)}
        \arrow["{U\circ}", from=1-1, to=1-2]
        \arrow["{G_1}"', from=1-1, to=2-1]
        \arrow["\alpha", shorten <=6pt, shorten >=6pt, Rightarrow, from=1-2, to=2-1]
        \arrow["{G_2}", from=1-2, to=2-2]
        \arrow["{F_U}"', from=2-1, to=2-2]
    \end{tikzcd}\]
    where $F_U$ is the sheaf composition functor associated to $U$. We construct a natural transformation $\alpha$ between the two composite functors as indicated in the diagram. Its components consist of natural transformations of functors $(U \circ P)^{sh} \to U \circ P^{sh}$ for each presheaf $P : \C^{\op} \to \A$, where $(-)^{sh}$ denotes sheafification. By adjunction, this is the same as giving a natural transformation $U \circ P \to U \circ P^{sh}$, which is obtained from whiskering $U$ with the unit of the sheafification adjunction applied to $P$. We omit the proof of naturality, and refer to the \Lean code for details.
\end{construction}

\begin{proposition}\label{prop:preservesSheafification_diagram_iso} \href{https://github.com/leanprover-community/mathlib4/blob/ffb2f0dbf80ad18fa75f85bd0facee6737b59acf/Mathlib/CategoryTheory/Sites/PreservesSheafification.lean#L198-L200}{\faExternalLink}
    Let $\C, J, \A, \B$, and $U$ be like in Construction~\ref{const:preservesSheafification_diagram}. Suppose in addition that $U$ preserves sheafification. Then the natural transformation $\alpha$ in Construction~\ref{const:preservesSheafification_diagram} is an isomorphism.
\end{proposition}
\begin{proof}
    This follows from the fact that $U$ preserves sheafification, we refer to the \Lean code for details.
\end{proof}

\begin{proposition}\label{prop:hasSheafify_concrete} 
    Let $\C$ be a large category equipped with a Grothendieck topology $J$. 
    \begin{enumerate}
        \item The category of large sets has sheafification with respect to $J$.
        \item Let $R$ be a ring. The category of large $R$-modules has sheafification with respect to $J$
    \end{enumerate}
    Suppose now that $\C$ is essentially small. 
    \begin{enumerate}
        \item[(3)] The category of small sets has sheafification with respect to $J$.
        \item[(4)] Let $R$ be a ring. The category of small $R$-modules has sheafification with respect to $J$
    \end{enumerate}
\end{proposition}
\begin{proof}
    This follows from the fact that these categories have limits and their forgetful functors preserve limits and filtered colimits. 
\end{proof}

\begin{proposition}\label{prop:hasSheafCompose_concrete}
    Let $\C$ be a large category equipped with a Grothendieck topology $J$. 
    \begin{itemize}
        \item Let $R$ be a ring. The forgetful functor from the category of large $R$-modules to the category of large sets has sheaf composition with respect to $J$. 
        \item If $\C$ is essentially small and $R$ is a ring, then the forgetful functor from the category of small $R$-modules to the category of small sets has sheaf composition with respect to $J$. 
    \end{itemize}
\end{proposition}
\begin{proof}
    This follows from the fact that this forgetful functor preserves limits. 
\end{proof}

\begin{proposition}\label{prop:preservesSheafification_concrete}
    Let $\C$ be a large category equipped with a Grothendieck topology $J$. 
    \begin{itemize}
        \item Let $R$ be a ring. The forgetful functor from the category of large $R$-modules to the category of large sets preserves sheafification. 
        \item If $\C$ is essentially small and $R$ is a ring, then the forgetful functor from the category of small $R$-modules to the category of small sets preserves sheafification. 
    \end{itemize}
\end{proposition}
\begin{proof}
    This follows from the fact that this forgetful functor preserves limits and filtered colimits, and reflects isomorphisms. 
\end{proof}

Propositions~\ref{prop:hasSheafify_concrete},~\ref{prop:hasSheafCompose_concrete}, and~\ref{prop:preservesSheafification_concrete} are not stated explicitly in this way in \mathlib. Instead, there are much more general instances for sheaves valued in concrete categories whose forgetful functor preserves certain limits and colimits. For the convenience of the reader, we link to files on a branch of \mathlib (one file for each Proposition), which show that \Lean automatically synthesizes these instances \href{https://github.com/leanprover-community/mathlib4/blob/42d44d43b9d4ce35d079ecee2ecd173eaf640091/test/CategoryTheory/Sites/ConcreteSheafification.lean}{\faExternalLink} \href{https://github.com/leanprover-community/mathlib4/blob/42d44d43b9d4ce35d079ecee2ecd173eaf640091/test/CategoryTheory/Sites/Whiskering.lean}{\faExternalLink} \href{https://github.com/leanprover-community/mathlib4/blob/42d44d43b9d4ce35d079ecee2ecd173eaf640091/test/CategoryTheory/Sites/PreservesSheafification.lean}{\faExternalLink}.

Proposition~\ref{prop:hasSheafify_concrete} shows that in the setting of condensed- and light condensed sets and modules, we can sheafify. Proposition~\ref{prop:hasSheafCompose_concrete} shows that there is a ``forgetful functor'' $\CondMod_R \to \CondSet$ which takes a condensed $R$-module to its ``underlying'' condensed set (and the analogue for light condensed objects). Proposition~\ref{prop:preservesSheafification_concrete} shows that this functor fits into a diagram like the one in Construction~\ref{const:preservesSheafification_diagram}, which commutes up to isomorphism, as the bottom row. 

\section{Constant sheaves}\label{sec:constant}
\begin{notation} \href{https://github.com/leanprover-community/mathlib4/blob/ffb2f0dbf80ad18fa75f85bd0facee6737b59acf/Mathlib/CategoryTheory/Sites/ConstantSheaf.lean#L65}{\faExternalLink}
    Let $\C$ be a category equipped with a Grothendieck topology $J$. Let $\A$ be a category which has weak sheafification with respect to $J$ and let $X$ be an object of $\A$. We denote by $\underline{X}$ the constant sheaf at $X$, i.e. the sheafification of the constant presheaf $\mathsf{cst}_X : \C^{\op} \to \A$, $Y \mapsto X$.
\end{notation}

\begin{construction}\label{const:constantSheafAdj} \href{https://github.com/leanprover-community/mathlib4/blob/ffb2f0dbf80ad18fa75f85bd0facee6737b59acf/Mathlib/CategoryTheory/Sites/ConstantSheaf.lean#L68-L71}{\faExternalLink}
    \emph{The constant sheaf adjunction}. Let $\C$ be a category equipped with a Grothendieck topology $J$. Let $\A$ be a category which has weak sheafification with respect to $J$. Suppose that $\C$ has a terminal object $*$ and let $\mathsf{ev}_*$ denote the functor which maps a sheaf or presheaf $F$ to its evaluation $F(*)$. We construct an adjunction 
    \[\underline{(-)} \dashv \mathsf{ev}_*\] 
    We do this by constructing an adjunction $\mathsf{cst}_{-} \dashv \mathsf{ev}_*$ and composing it with the sheafification adjunction. The unit of the desired adjunction comes from the fact that the constant presheaf followed by evaluation is simply the identity. To define the counit, we need to give a natural transformation from the functor given by evaluation at the point followed by the constant presheaf functor, to the identity functor. Its component at a presheaf $F$ corresponds to a natural transformation from the constant presheaf at $F(*)$ to $F$, and the component of that natural transformation at an object $X$ is just a map $F(*) \to F(X)$, which is induced by the unique map $X \to *$. Naturality is easy.
\end{construction}

\begin{definition}\label{def:discreteSheaf} \href{https://github.com/leanprover-community/mathlib4/blob/ffb2f0dbf80ad18fa75f85bd0facee6737b59acf/Mathlib/CategoryTheory/Sites/ConstantSheaf.lean#L80-L81}{\faExternalLink}
    Let $\C$ be a category equipped with a Grothendieck topology $J$. Let $\A$ be a category which has weak sheafification with respect to $J$ (this property ensures that the constant sheaf functor $\underline{(-)} : \A \to \Sh_{J}(\C, \A)$ exists). We say that an $\A$-valued sheaf $\F$ on $\C$ is \emph{constant} if it is in the essential image of the constant sheaf functor, i.e.~if there exists an object $X$ of $A$ and an isomorphism $\underline{X} \cong \F$.
\end{definition}

\begin{proposition}\label{prop:isDiscrete_iff_mem_essImage} \href{https://github.com/leanprover-community/mathlib4/blob/ffb2f0dbf80ad18fa75f85bd0facee6737b59acf/Mathlib/CategoryTheory/Sites/ConstantSheaf.lean#L113-L116}{\faExternalLink}
    Let $\C$, $J$, and $\A$, be like in Definition~\ref{def:discreteSheaf}. Suppose that $\C$ has a terminal object $*$ and that the constant sheaf functor $\A \to \Sh_{J}(\C, \A), X \mapsto \underline{X}$ is fully faithful. A sheaf $\F$ on $\C$ is constant if and only if  the counit induces an isomorphism $\underline{\F(*)} \to \F$.
\end{proposition}
\begin{proof}
    This is a direct application of Proposition~\ref{prop:adjunction-generality-1}.
\end{proof}

The main result of this section is Proposition~\ref{prop:isDiscrete_iff_forget}. To prove it, we first need a technical lemma:

\begin{lemma}\label{lemma:constantSheafAdj_counit_w} \href{https://github.com/leanprover-community/mathlib4/blob/ffb2f0dbf80ad18fa75f85bd0facee6737b59acf/Mathlib/CategoryTheory/Sites/ConstantSheaf.lean#L193-L208}{\faExternalLink}
    Let $\C$ be a category equipped with a Grothendieck topology $J$, let $\A$ and $\B$ be categories with weak sheafification with respect to $J$, and let $U : \A \to \B$ be a functor which preserves sheafification and such that $U$ has sheaf composition with respect to $J$. Suppose that the constant sheaf functors $\A \to \Sh_J(\C, \A)$ and $\B \to \Sh_J(\C, \B)$ are both fully faithful, and that $\C$ has a terminal object $*$. Let $F$ be an $\A$-valued sheaf on $\C$. Denote by $\epsilon^{\A}$ the counit of the constant sheaf adjunction between $\A$ and $\Sh_J(\C, \A)$ and $\epsilon^{\B}$ the counit of the constant sheaf between $\B$ and $\Sh_J(\C, \B)$. Let $F_U$ denote the sheaf composition functor associated to $U$. There is a commutative diagram
    \[\begin{tikzcd}
        {\underline{U(F(*))}} \\
        {U \circ \underline{F(*)}} & {U \circ F}
        \arrow[from=1-1, to=2-1]
        \arrow["{\epsilon^{\B}_{U\circ F}}", from=1-1, to=2-2]
        \arrow["{F_U(\epsilon^{\A}_F)}"', from=2-1, to=2-2]
    \end{tikzcd},\]
    where the vertical arrow is an isomorphism.
\end{lemma}
\begin{proof}
    We start by extending the diagram in Construction~\ref{const:preservesSheafification_diagram} by the constant presheaf functors: 
    \[\begin{tikzcd}
        \A & \B \\
        {\PSh(\C,\A)} & {\PSh(\C, \B)} \\
        {\Sh_J(\C,\A)} & {\Sh_J(\C,\B)}
        \arrow["U", from=1-1, to=1-2]
        \arrow["{\mathsf{cst}_{-}}"', from=1-1, to=2-1]
        \arrow["{\mathsf{cst}_{U(-)}}", from=1-2, to=2-2]
        \arrow["{U\circ}", from=2-1, to=2-2]
        \arrow["{G_1}"', from=2-1, to=3-1]
        \arrow["{G_2}", from=2-2, to=3-2]
        \arrow["{F_U}"', from=3-1, to=3-2]
    \end{tikzcd}\]
    The upper square is strictly commutative by definition, and the lower square is commutative up to isomorphism by Proposition \ref{prop:preservesSheafification_diagram_iso}. For any $X \in \A$ we obtain an isomorphism $U \circ \underline{X} \cong \underline{U(X)}$ by applying this natural isomorphism at the object $X$. In particular, by taking $X$ as $F(*)$, we obtain the vertical arrow in the triangle in the lemma statement. 

    To verify that the triangle commutes, it suffices to check the equality after precomposing with the unit of the sheafification adjunction (this amounts to applying the hom set equivalence of the adjunction). After precomposing with the unit, it becomes a question of unfolding the definitions of the maps to see that it commutes.
\end{proof}

\begin{proposition}\label{prop:isDiscrete_iff_forget} \href{https://github.com/leanprover-community/mathlib4/blob/ffb2f0dbf80ad18fa75f85bd0facee6737b59acf/Mathlib/CategoryTheory/Sites/ConstantSheaf.lean#L227-L231}{\faExternalLink}
    Let $\C$ be a category equipped with a Grothendieck topology $J$, let $\A$ and $\B$ be categories with weak sheafification with respect to $J$, and let $U : \A \to \B$ be a functor which preserves sheafification and such that $U$ has sheaf composition with respect to $J$. Suppose that the constant sheaf functors $\A \to \Sh_J(\C, \A)$ and $\B \to \Sh_J(\C, \B)$ are both fully faithful, and that $\C$ has a terminal object $*$. Let $F$ be an $\A$-valued sheaf on $\C$. Then $F$ is constant if and only if $U\circ F$ is constant.
\end{proposition}
\begin{proof}
    With notation as in Lemma~\ref{lemma:constantSheafAdj_counit_w}, this amounts to showing that 
    \[\epsilon^{\A}_{F} : \underline{F(*)} \to F\]
    is an isomorphism if and only if
    \[\epsilon^{\B}_{U \circ F} : \underline{U(F(*))} \to U \circ F\]
    is an isomorphism. We know that $F_U$ reflects isomorphisms by Proposition \ref{prop:reflectsIso_etc}, so we conclude using Lemma~\ref{lemma:constantSheafAdj_counit_w}.
\end{proof}

\section{The functors from $\Set$ to $\CondSet$}\label{sec:functors}
The constant sheaf functor 
\begin{align*}
    \underline{(-)} : \Set & \to \CondSet
\end{align*}
is left adjoint to $U : \CondSet \to \Set$, the functor which maps a condensed set $X$ to the underlying set $X(*)$, by Construction~\ref{const:constantSheafAdj}. Recall that a condensed set $X$ is \emph{discrete} if it is in the essential image of this functor, i.e.~if there exists a set $X'$ and an isomorphism $X \cong \underline{X'}$.

The goal of this section is to construct an isomorphism between $\underline{(-)}$ and the functor 
\begin{align*}
    L : \Set & \to \CondSet \\
    X & \mapsto \LocConst(-, X).
\end{align*}
This will yield a proof of Theorem~\ref{thm:A2}.  
\begin{maintheorembody}\label{thm:A2}
    A condensed set is discrete if and only if it is in the essential image of the functor 
    \[
        L : \Set \to \CondSet,
    \] 
    i.e. if there exists some set $X'$ and an isomorphism of condensed sets
    \[
        X \cong \LocConst(-,X')
    \]
\end{maintheorembody}

The desired isomorphism is obtained by constructing an adjunction $L \dashv U$, and using the fact that adjoints are unique up to isomorphism. 

The unit is easy to define (see Construction~\ref{const:unit}) and obviously an isomorphism, giving full faithfulness once we have established the adjunction. Defining the counit is the hard part (see Construction~\ref{const:counit} and \S\ref{sec:counit_naturality}).

\begin{construction}\label{const:unit} \href{https://github.com/leanprover-community/mathlib4/blob/ffb2f0dbf80ad18fa75f85bd0facee6737b59acf/Mathlib/Condensed/Discrete/LocallyConstant.lean#L292-L300}{\faExternalLink}
    \emph{Unit}. We construct a natural transformation 
    \[
        \mathsf{Id}_{\Set} \to U \circ L.
    \] 
    This amounts to giving a map $X \to \LocConst(*, X)$ for every set $X$ and proving naturality in $X$. The desired map is just the one which takes an element $x \in X$ to the corresponding constant map. Naturality is easy, and we also easily see that this natural transformation is in fact an isomorphism.
\end{construction}

\begin{construction}\label{const:counit} \href{https://github.com/leanprover-community/mathlib4/blob/ffb2f0dbf80ad18fa75f85bd0facee6737b59acf/Mathlib/Condensed/Discrete/LocallyConstant.lean#L136-L151}{\faExternalLink}
    \emph{Components of the counit}. Let $X$ be a condensed set and let $S$ be a compact Hausdorff space. We construct a map
    \[
        \LocConst(S, X(*)) \to X(S)
    \]
    as follows: A locally constant map $f : S \to X(*)$ corresponds to a finite partition of $S$ into closed (and hence compact Hausdorff) subsets $S_1, \dots, S_n$ (these are the fibres of $f$). We have isomorphisms 
    \begin{align*}
        X(S) \cong X\left(S_1 \amalg \cdots \amalg S_n\right) \cong X(S_1) \times \cdots \times X(S_n),
    \end{align*}
    so giving an element of $X(S)$ is the same as giving an element of each of $X(S_1), \dots, X(S_n)$. Let $x_i \in X(*)$ denote the element such that $S_i = f^{-1}\{x_i\}$, and let $g_i : S_i \to *$ be the unique map. Then our element of $X(S_i)$ will be $X(g_i)(x_i)$. 
\end{construction}

\subsection{Naturality of the counit}\label{sec:counit_naturality}
We need to show that the map defined in Construction~\ref{const:counit} as the components of the counit really defines a natural transformation of functors $L \circ U \to \mathsf{Id}_{\CondSet}$, where $U : \CondSet \to \Set$ is the underlying set functor. We start with naturality in the compact Hausdorff space $S$ and then proceed to show that it is natural in the condensed set $X$. This is where it starts to really matter how one sets up the formalized proof. A key lemma which easy to prove and is used repeatedly in these naturality proofs is Lemma~\ref{lemma:presheaf_ext}.

\begin{lemma}\label{lemma:presheaf_ext} \href{https://github.com/leanprover-community/mathlib4/blob/ffb2f0dbf80ad18fa75f85bd0facee6737b59acf/Mathlib/Condensed/Discrete/LocallyConstant.lean#L158-L167}{\faExternalLink}
    Let $X$ and $Y$ be condensed sets, $S$ a compact Hausdorff space and $f : S \to Y(*)$ a locally constant map. Let 
    \[
        S \cong S_1 \amalg \cdots \amalg S_n
    \]
    be the corresponding decomposition of $S$ into the fibres of $f$. Denote by $\iota_i : S_i \to S$ the inclusion map. Let $x, y \in X(S)$ and suppose that for all $i$, 
    \[
        X(\iota_i)(x) = X(\iota_i)(y).
    \] 
    Then $x = y$.
\end{lemma}
\begin{proof}
    The condition simply says that $x$ and $y$ are equal when considered as elements of the product 
    \[ 
        X(S_1) \times \cdots \times X(S_n).
    \]
\end{proof}

To be able to successfully use this lemma in the formalization, it is important to be careful with setting everything up. In \Lean, the statement is as follows:
\begin{lstlisting}
lemma presheaf_ext (X : (CompHausLike.{u} P)ᵒᵖ ⥤ Type max u w)
    [PreservesFiniteProducts X] (x y : X.obj ⟨S⟩)
    [HasExplicitFiniteCoproducts.{u} P]
    (h : ∀ (a : Fiber f), X.map (sigmaIncl f a).op x = X.map (sigmaIncl f a).op y) : 
    x = y := ...
\end{lstlisting}
The assumption is literally phrased such that one needs to check that for every fibre $f^{-1}\{x\}$ (denoted by \lean{a : Fiber f} in the \Lean statement), $X$ applied to the inclusion map $f^{-1}\{x\} \to S$ agrees on the two elements. The terms of type \lean{Fiber f} are defined as actual subtypes of $S$ and the maps \lean{sigmaIncl f a} are the inclusion maps.

\begin{lemma}\label{lemma:counit_natural_S} \href{https://github.com/leanprover-community/mathlib4/blob/ffb2f0dbf80ad18fa75f85bd0facee6737b59acf/Mathlib/Condensed/Discrete/LocallyConstant.lean#L200-L216}{\faExternalLink}
    The proposed counit defined in Construction~\ref{const:counit} is natural in the compact Hausdorff space $S$. 
\end{lemma}
\begin{proof}
    Fix a condensed set $X$. Given a compact Hausdorff space $S$ we denote by $\epsilon_S$ the map 
    \[
        \LocConst(S, X(*)) \to X(S)
    \]
    defined in Construction~\ref{const:counit}. We need to show that for every continuous map $g : T \to S$ of compact Hausdorff spaces, the diagram
    \[\begin{tikzcd}
        {\LocConst(S, X(*))} & {X(S)} \\
        {\LocConst(T,X(*))} & {X(T)}
        \arrow["{\epsilon_S}", from=1-1, to=1-2]
        \arrow["{\circ g}"', from=1-1, to=2-1]
        \arrow["{X(g)}", from=1-2, to=2-2]
        \arrow["{\epsilon_T}"', from=2-1, to=2-2]
    \end{tikzcd}\]
    commutes. Let $f : S \to X(*)$ be a locally constant map. We need to show that 
    \[
        X(g)(\epsilon_S(f)) = \epsilon_T(f \circ g).
    \]
    We apply Lemma \ref{lemma:presheaf_ext} with $Y = X$ to the locally constant map $f \circ g : T \to X(*)$. Let 
    \[
        T \cong T_1 \amalg \cdots \amalg T_n
    \]
    be the decomposition of $T$ into the fibres of $f \circ g$ and denote by $\iota^T_i : T_i \to T$ the inclusion maps. Now we need to show that 
    \[
        X(\iota^T_i) \left(X(g)(\epsilon_S(f))\right) = 
        X(\iota^T_i)\left(\epsilon_T(f \circ g)\right).
    \]
    Let $x_i \in X(*)$ be the element of which $T_i$ is the fibre and let $t_i : T_i \to *$ denote the unique map. Then the above simplifies to 
    \[
        X(g \circ \iota^T_i)\left(\epsilon_S(f)\right) = X(t_i)(x_i). 
    \]
    Now, let 
    \[
        S \cong S_1 \amalg \cdots \amalg S_m
    \]
    be the decomposition of $S$ into the fibres of $f$. Denote by $\iota^S_j : S_j \to S$ the inclusion maps. Let $j$ be such that $S_j = f^{-1}(x_i)$. Then $g$ restricts to a continuous map $g' : T_i \to S_j$, and we have $g \circ \iota^T_i = \iota^S_j \circ g'$. Letting $s_j : S_j \to *$ denote the unique map, we now have
    \begin{align*}
        X(g \circ \iota^T_i)\left(\epsilon_S(f)\right) 
            &= X(g')\left(X(\iota^S_j)(\epsilon_S(f))\right) \\
            &= X(g')\left(X(s_j)(x_i)\right) \\ 
            &= X(s_j \circ g')(x_i) \\
            &= X(t_i)(x_i)
    \end{align*}
    as desired.
\end{proof}

\begin{lemma}\label{lemma:counit_natural_X} \href{https://github.com/leanprover-community/mathlib4/blob/ffb2f0dbf80ad18fa75f85bd0facee6737b59acf/Mathlib/Condensed/Discrete/LocallyConstant.lean#L250-L287}{\faExternalLink}
    The proposed counit defined in Construction~\ref{const:counit} is natural in the condensed set $X$.
\end{lemma}
\begin{proof}
    Let $S$ be a compact Hausdorff space. Given any condensed set $X$, denote by $\epsilon_X$ the map 
    \[
        \LocConst(S, X(*)) \to X(S)
    \]
    defined in Construction~\ref{const:counit}. Let $g : X \to Y$ be a morphism of condensed sets. We need to show that the diagram
    \[\begin{tikzcd}
        {\LocConst(S, X(*))} & {X(S)} \\
        {\LocConst(S,Y(*))} & {Y(S)}
        \arrow["{\epsilon_X}", from=1-1, to=1-2]
        \arrow["{g_*\circ}"', from=1-1, to=2-1]
        \arrow["{g_S}", from=1-2, to=2-2]
        \arrow["{\epsilon_Y}"', from=2-1, to=2-2]
    \end{tikzcd}\]
    commutes (here, $g_T$ denotes the component of the natural transformation $g$ at a compact Hausdorff space $T$). Let $f : S \to X(*)$ be a locally constant map. We need to show that 
    \[
        g_S (\epsilon_X(f)) = \epsilon_Y(g_* \circ f).
    \]
    We apply Lemma \ref{lemma:presheaf_ext} with $X = Y$ to the locally constant map $g_* \circ f : S \to Y(*)$. Let 
    \[
        S \cong S_1 \amalg \cdots \amalg S_n
    \]
    be the decomposition of $S$ into the fibres of $g_* \circ f$ and denote by $\iota_i : S_i \to S$ the inclusion maps. Now we need to show that 
    \[
        Y(\iota_i) \left(g_S (\epsilon_X(f))\right) = Y(\iota_i)\left(\epsilon_Y(g_* \circ f)\right).
    \]
    Let $y_i \in Y(*)$ be the element of which $S_i$ is the fibre and let $s_i : S_i \to *$ denote the unique map. Then the above simplifies to 
    \[
        Y(\iota_i) \left(g_S (\epsilon_X(f))\right) = Y(s_i)(y_i).
    \]
    This is a question of proving equality of two elements of the set $Y(S_i)$. We apply Lemma \ref{lemma:presheaf_ext} again, this time to the locally constant map $f \circ \iota_i : S_i \to X(*)$ (and the roles of $X$ and $Y$ swapped with respect to the names of the condensed sets in the lemma statement). Let 
    \[
        S_i \cong S_{i, 1} \amalg \cdots \amalg S_{i, m}
    \]
    be the decomposition of $S_i$ into the fibres of $f \circ \iota_i$ and denote by $\iota_{i, j} : S_{i, j} \to S_i$ the inclusion maps and by $s_{i, j} : S_{i, j} \to *$ the unique map. Then we need to show that 
    \[
        Y(\iota_i \circ \iota_{i, j}) \left(g_S (\epsilon_X(f))\right) = Y(s_{i, j})(y_i).
    \]
    By naturality of $g$ the goal now becomes
    \[
        g_{S_{i, j}}(X(\iota_i \circ \iota_{i, j})(\epsilon_X(f))) = Y(s_{i, j})(y_i)
    \]
    Now, let
    \[
        S \cong S'_1 \amalg \cdots \amalg S'_r
    \]
    be the decomposition of $S$ into the fibres of $f$. Denote by $\iota'_{k} : S'_{k} \to S$ the inclusion maps. There exists a $k$ such that $S_{i, j} \subseteq S'_k$, denote the inclusion map by $\iota'$. Let $x'_k \in X(*)$ be the element of which $S'_k$ is the fibre, and denote by $s'_k : S'_k \to *$ the unique map. Then we have
    \begin{align*}
        X(\iota_i \circ \iota_{i, j})(\epsilon_X(f))
            &= X(\iota'_{k} \circ \iota')(\epsilon_X(f)) \\
            &= X(\iota')(X(\iota'_k)(\epsilon_X(f))) \\
            &= X(\iota')(X(s'_k)(x'_k)) \\ 
            &= X(s_{i, j})(x'_k)
    \end{align*}
    The goal now becomes
    \[
        g_{S_{i, j}}(X(s_{i, j})(x'_k)) = Y(s_{i, j})(y_i)
    \]
    Applying naturality of $g$ again, it reduces to
    \[
        Y(s_{i, j})(g_*(x'_k)) = Y(s_{i, j})(y_i),
    \]
    so it suffices to show that $g_*(x'_k) = y_i$. Let $s \in S_{i, j}$ be some element. Then since $s \in S'_k$, we have $f(s) = x'_k$. Thus, since $s \in S_i$, we have $g_*(x'_k) = g_*(f(s)) = y_i$, as desired.
\end{proof}

\subsection{Triangle identities}
To prove that the natural transformations defined in Constructions~\ref{const:unit}, \ref{const:counit} form the unit and counit of the desired adjunction, we need to verify the triangle identities. Throughout this subsection, we let $\eta_X$ denote the component of the unit at a set $X$ and $\epsilon_{X, S}$ the component of the counit at a condensed set $X$ and compact Hausdorff space $S$. Recall the notation $L : \Set \to \CondSet$ for the functor mapping a set $X$ to the sheaf of locally constant maps to $X$ and $U : \CondSet \to \Set$ for the underlying set functor, mapping a condensed set $X$ to $X(*)$.

\begin{lemma}\label{lemma:left_triangle} \href{https://github.com/leanprover-community/mathlib4/blob/ffb2f0dbf80ad18fa75f85bd0facee6737b59acf/Mathlib/Condensed/Discrete/LocallyConstant.lean#L329-L336}{\faExternalLink}
    Let $X$ be a set. We have
    \[
        \epsilon_{L(X)} \circ L(\eta_X) = {\mathsf{id}}_{L(X)}
    \]
\end{lemma}
\begin{proof}
    We need to show that for every compact Hausdorff space $S$ and every locally constant map $f : S \to X$, 
    \[
        \epsilon_{L(X), S}(\eta_X \circ f) = f.
    \]
    We apply Lemma~\ref{lemma:presheaf_ext} with $X$ and $Y$ both set to $L(X)$, to the locally constant map $\eta_X \circ f : S \to \LocConst(*, X)$. Let
    \[
        S \cong S_1 \amalg \cdots \amalg S_n
    \]
    be the decomposition of $S$ into the fibres of $\eta_X \circ f$ and denote by $\iota_i : S_i \to S$ the inclusion maps. Now we need to show that 
    \[
        L(X)(\iota_i) \left(\epsilon_{L(X), S}(\eta_X \circ f)\right) = L(X)(\iota_i)(f)
    \]
    Let $x_i \in \LocConst(*, X)$ be the element of which $S_i$ is the fibre and let $s_i : S_i \to *$ denote the unique map. Then the above simplifies to 
    \[
        x_i \circ s_i = f \circ \iota_i.
    \]
    Let $s \in S_i$ be arbitrary. We have 
    \[
        x_i(s_i(s)) = f(s) = f(\iota_i(s))
    \]
    as desired.
\end{proof}

\begin{lemma}\label{lemma:right_triangle} \href{https://github.com/leanprover-community/mathlib4/blob/ffb2f0dbf80ad18fa75f85bd0facee6737b59acf/Mathlib/Condensed/Discrete/LocallyConstant.lean#L337-L352}{\faExternalLink}
    Let $X$ be a condensed set. We have
    \[
        U(\epsilon_X) \circ \eta_{U(X)} = \mathsf{id}_{U(X)}
    \]
\end{lemma}
\begin{proof}
    Let $x \in X(*)$. We need to show that 
    \[
        \epsilon_{X, *} (\eta_{X(*)} (x)) = x.
    \]
    This is trivial.
\end{proof}

\begin{remark}
    The proof of Lemma~\ref{lemma:right_triangle} in \Lean actually uses Lemma~\ref{lemma:presheaf_ext}, trivially decomposing $*$ into a disjoint union corresponding to the locally constant map $\eta_{X(*)}(x)$. This shows how important Lemma~\ref{lemma:presheaf_ext} is in proving anything at all about the values of the counit. 
\end{remark}

\subsection{The adjunction}
\begin{construction}\label{const:adjunction} \href{https://github.com/leanprover-community/mathlib4/blob/ffb2f0dbf80ad18fa75f85bd0facee6737b59acf/Mathlib/Condensed/Discrete/LocallyConstant.lean#L324-L352}{\faExternalLink}
    All the above assembles into an adjunction
    \[
        L \dashv U,
    \]
    with unit defined in Construction~\ref{const:unit}, counit given by Construction~\ref{const:counit} together with Lemmas~\ref{lemma:counit_natural_S} and~\ref{lemma:counit_natural_X}, and triangle identities by Lemmas~\ref{lemma:left_triangle} and \ref{lemma:right_triangle}.
\end{construction}

\begin{construction}\label{const:L_iso_D} \href{https://github.com/leanprover-community/mathlib4/blob/ffb2f0dbf80ad18fa75f85bd0facee6737b59acf/Mathlib/Condensed/Discrete/LocallyConstant.lean#L376-L377}{\faExternalLink}
    By uniqueness of adjoints, we obtain a natural isomorphism
    \[
        L \cong \underline{(-)},
    \]
    where $\underline{X}$ denotes the constant sheaf at a set $X$.
\end{construction}

\begin{proposition} \href{https://github.com/leanprover-community/mathlib4/blob/ffb2f0dbf80ad18fa75f85bd0facee6737b59acf/Mathlib/Condensed/Discrete/LocallyConstant.lean#L380-L389}{\faExternalLink}
    The functor $L : \Set \to \CondSet$ (and hence also $\underline{(-)}$ because of Construction~\ref{const:L_iso_D}) is fully faithful.
\end{proposition}
\begin{proof}
    This follows from the fact that the unit of the adjunction in Construction~\ref{const:adjunction} is an isomorphism.
\end{proof}

\subsection{Generality}\label{subsec:generality}
The formalization of the results described in this section was in fact mostly done for sheaves on a site of the form \lean{CompHausLike P} as described in \S\ref{sec:size}, and only at the very end specialized to condensed sets and light condensed sets. 

Some conditions on the predicate \lean{P} were required. A few were required simply for the coherent topology to exist, but the two most important ones were
\begin{lstlisting}
[∀ (S : CompHausLike.{u} P) (p : S → Prop), HasProp P (Subtype p)],
\end{lstlisting}
and
\begin{lstlisting}
[HasExplicitFiniteCoproducts P].
\end{lstlisting}
The first one says that every subtype of an object of the category \lean{CompHausLike P} satisfies the predicate \lean{P}. This could in fact be weakened to only requiring clopen subsets to satisfy it. This more restrictive assumption turned out to be more convenient to work with, because the weaker one required carrying around too many proofs that certain subsets were clopen. It was still good enough, because in the application, \lean{P} is either \lean{True}, or the predicate saying that the space is totally disconnected and second countable. These are both stable under taking subspaces. This is an example of the importance of finding the ``correct generality'' for formalized statements. One wants the results to hold in the greatest generality possible, while not making an unnecessary effort for more generality when the payoff is small.

The second one says that finite disjoint unions of spaces in the category \lean{CompHausLike P} exist, and form a coproduct. 

The purpose of both these assumptions is explained by the discussion following Lemma~\ref{lemma:presheaf_ext}. We need to be able to write an object of this category as a coproduct of the fibres of a locally constant map out of it, and we want the components of this coproduct to be literal subtypes of it.

\section{The colimit characterization}\label{sec:colimit}
Informally, one can state the main result of this section as follows: a condensed set $X$ is discrete if and only if for every profinite set $S = \varprojlim_i S_i$, the canonical map $\varinjlim_i X(S_i) \to X(S)$ is an isomorphism. One needs to clarify, however, what the phrase ``for every profinite set $S = \varprojlim_i S_i$'' means. This is done in Construction~\ref{const:asLimitCone} and Proposition~\ref{prop:asLimit} below. 

The condition that the canonical map $\varinjlim_i X(S_i) \to X(S)$ is an isomorphism is phrased in \Lean as saying that a cocone is colimiting (this is equivalent to the statement about the canonical map being an isomorphism, but more convenient in formalization as explained in Remark~\ref{rem:cocone_vs_informal} below). We say that a presheaf $X$ on $\Profinite$ satisfies \emph{the colimit condition} if this holds (Definition~\ref{def:colimit_condition}). 

\subsection{Profinite sets as limits}\label{subsec:profinite_as_limit}
The material in this section was mostly formalized by Adam Topaz and Calle Sönne, and was already in \mathlib before the start of this project. We do not write out informal proofs and refer instead to the \Lean code for details.

\begin{definition}\label{def:discrete_quotient} \href{https://github.com/leanprover-community/mathlib4/blob/ffb2f0dbf80ad18fa75f85bd0facee6737b59acf/Mathlib/Topology/DiscreteQuotient.lean#L69-L72}{\faExternalLink} 
\href{https://github.com/leanprover-community/mathlib4/blob/ffb2f0dbf80ad18fa75f85bd0facee6737b59acf/Mathlib/Topology/DiscreteQuotient.lean#L102}{\faExternalLink} 
\href{https://github.com/leanprover-community/mathlib4/blob/ffb2f0dbf80ad18fa75f85bd0facee6737b59acf/Mathlib/Topology/DiscreteQuotient.lean#L110-L111}{\faExternalLink}
    Let $X$ be a topological space. A \emph{discrete quotient} of $X$ is an equivalence relation on $X$ with open fibres. Let $r$ be a discrete quotient on $X$ and denote by $X_r$ its set of equivalence classes. We obtain a \emph{projection map} $\pi_r : X \to X_r$, $x \mapsto [x]$. Give $X_r$ the discrete topology. Then the projection map $\pi_r : X \to X_r$ is a quotient map.
\end{definition}

\begin{construction}\label{const:dq_le} \href{https://github.com/leanprover-community/mathlib4/blob/ffb2f0dbf80ad18fa75f85bd0facee6737b59acf/Mathlib/Topology/DiscreteQuotient.lean#L137-L145}{\faExternalLink} \href{https://github.com/leanprover-community/mathlib4/blob/ffb2f0dbf80ad18fa75f85bd0facee6737b59acf/Mathlib/CategoryTheory/Filtered/Basic.lean#L522-L523}{\faExternalLink}
    Let $r, s$ be two discrete quotients of a topological space $X$. We say that $r \leq s$ if for all $x, y \in X$, $r(x, y) \implies s(x,y)$. This gives the set of discrete quotients of $X$ the structure of an inf-semilattice, i.e.~a partial order with greatest lower bound. This makes the category structure inherited from this order cofiltered. 
\end{construction}

\begin{remark}\label{remark:dq_le} \href{https://github.com/leanprover-community/mathlib4/blob/ffb2f0dbf80ad18fa75f85bd0facee6737b59acf/Mathlib/Topology/DiscreteQuotient.lean#L211-L213}{\faExternalLink}
    One can interpret the order on the discrete quotients as follows. Using the same notation as in Definition~\ref{def:discrete_quotient} and Construction~\ref{const:dq_le}, if $r \leq s$, then there exists a map 
    $X_r \to X_s$ making the triangle 
    \[\begin{tikzcd}
        & X \\
        {X_r} && {X_s}
        \arrow["{\pi_r}"', from=1-2, to=2-1]
        \arrow["{\pi_s}", from=1-2, to=2-3]
        \arrow[from=2-1, to=2-3]
    \end{tikzcd}\]
    commute.
\end{remark}

\begin{construction}\label{const:fintypeDiagram} \href{https://github.com/leanprover-community/mathlib4/blob/ffb2f0dbf80ad18fa75f85bd0facee6737b59acf/Mathlib/Topology/Category/Profinite/AsLimit.lean#L43-L48}{\faExternalLink}
    Let $S$ be a profinite set. We define a functor from the category of discrete quotients of $S$ to the category of finite sets by mapping a discrete quotient $i$ of $S$ to the set $S_i$ of equivalence classes (this lands in finite sets because $S$ is compact), and a morphism $i \leq j$ to the map $S_i \to S_j$ described in Remark~\ref{remark:dq_le}.
\end{construction}

\begin{construction}\label{const:asLimitCone} \href{https://github.com/leanprover-community/mathlib4/blob/ffb2f0dbf80ad18fa75f85bd0facee6737b59acf/Mathlib/Topology/Category/Profinite/AsLimit.lean#L55-L57}{\faExternalLink}
    We define a cone on the functor $i \mapsto S_i$ described in Construction~\ref{const:fintypeDiagram} in the category $\Profinite$. The cone point is the profinite set $S$ and the projection maps are the projections $\pi_i : S \to S_i$ defined in Definition~\ref{def:discrete_quotient}. 
\end{construction}

\begin{proposition}\label{prop:asLimit} \href{https://github.com/leanprover-community/mathlib4/blob/ffb2f0dbf80ad18fa75f85bd0facee6737b59acf/Mathlib/Topology/Category/Profinite/AsLimit.lean#L91-L92}{\faExternalLink}
    The cone described in Construction~\ref{const:asLimitCone} is limiting. 
\end{proposition}

\begin{lemma}\label{lemma:factor_through} \href{https://github.com/leanprover-community/mathlib4/blob/ffb2f0dbf80ad18fa75f85bd0facee6737b59acf/Mathlib/Topology/Category/Profinite/CofilteredLimit.lean#L189-L224}{\faExternalLink}
    Let $f : S \to \varprojlim_i T_i$ be a locally constant map from a profinite set $S$ to a cofiltered limit of profinite sets $T_i$. Then $f$ factors through one of the projections $\pi_i : \varprojlim_i T_i \to T_i$. 
\end{lemma}

\subsection{The colimit condition}

When we talk about ``a profinite set $S = \varprojlim_i S_i$'', we mean that $S$ is written as the limit of its discrete quotients as described in \S\ref{subsec:profinite_as_limit}. 

\begin{construction}\label{const:cocone}
    Let $X$ be a condensed set and let $S$ be a profinite set. We define a cocone on the functor $i \mapsto X(S_i)$ from the opposite category of discrete quotients of $S$ to $\Set$. The cocone point is $X(S)$ and the coprojections are given by the maps $X(\pi_i) : X(S_i) \to X(S)$. 
\end{construction}

\begin{definition}\label{def:colimit_condition}
    A presheaf of sets $X$ on $\Profinite$ \emph{satisfies the colimit condition} if for every profinite set $S$, the cocone described in Construction \ref{const:cocone} is colimiting.
\end{definition}

Construction~\ref{const:cocone} and Definition~\ref{def:colimit_condition} do not exist as explicit declarations in \mathlib. The colimit condition is expressed as the assumption
\begin{lstlisting}
∀ S : Profinite, IsColimit <| F.mapCocone S.asLimitCone.op
\end{lstlisting}
where \lean{S.asLimitCone.op} is the cocone obtained by taking the opposite of the cone in Construction~\ref{const:asLimitCone}, and \lean{F} is any presheaf on $\Profinite$. 

\begin{remark}\label{rem:cocone_vs_informal}
It is clear that a condensed set $X$ satisfies the colimit condition if and only if for every profinite set $S = \varprojlim_i S_i$ the canonical map $\varinjlim_i X(S_i) \to X(S)$ is an isomorphism (indeed, defining said ``canonical map'' requires the data of a cocone with cocone point $X(S)$ on the functor). This is closer to the way one would state the condition in informal mathematics. In formalized mathematics, it is often more convenient to work directly with cones and cocones. 
\end{remark}

We can now state the main theorem of this section more concisely and more precisely.

\begin{maintheorembody}\label{thm:main_colimit} \href{https://github.com/leanprover-community/mathlib4/blob/ffb2f0dbf80ad18fa75f85bd0facee6737b59acf/Mathlib/Condensed/Discrete/Characterization.lean#L74-L104}{\faExternalLink}
    A condensed set $X$ is discrete if and only if it satisfies the colimit condition.
\end{maintheorembody}

\subsection{Proof}\label{subsec:proof} 

The forward direction \href{https://github.com/leanprover-community/mathlib4/blob/ffb2f0dbf80ad18fa75f85bd0facee6737b59acf/Mathlib/Condensed/Discrete/Colimit.lean#L73-L75}{\faExternalLink} follows easily from material that was already in mathlib prior to this work. The argument is the following. Suppose $X$ is discrete. Then there is a set $Y$ and an isomorphism $X \cong \LocConst(-, Y)$. What one needs to show is that given a profinite set $S = \varprojlim_i S_i$, each locally constant map $f : S \to Y$, factors through a projection map $\pi_i : S \to S_i$, and if it factors through both $\pi_i$ and $\pi_j$, then there exists a $k$ with $\pi_k$ larger than both $\pi_i$ and $\pi_j$, such that $f$ factors through $\pi_k$. This follows from Lemma~\ref{lemma:factor_through} and surjectivity of the projection maps.

The difficult part is the other direction. Suppose that $X$ is a finite-product preserving presheaf on $\Profinite$ satisfying the colimit condition. It suffices to show that $X$ is isomorphic to the presheaf $\LocConst(-, X(*))$. Applying this to condensed sets, we get the characterization that a condensed set $X$ is discrete if and only if its underlying presheaf on $\Profinite$ satisfies the colimit condition, as desired.

It is not enough to give pointwise isomorphisms $X(S) \cong \LocConst(S, X(*))$ for each fixed profinite set $S$; we need to give such isomorphisms that are natural in $S$. As an intermediate step, we try to make the assignment $S \mapsto \varinjlim_i X(S_i)$ functorial in the profinite set $S$. The closest we can get to this is by using Kan extensions. 

The strategy now is to construct a sequence of isomorphisms as indicated \href{https://github.com/leanprover-community/mathlib4/blob/ffb2f0dbf80ad18fa75f85bd0facee6737b59acf/Mathlib/Condensed/Discrete/Colimit.lean#L248-L253}{\faExternalLink}
\[
    X \cong \Lan_{\iota^{\op}} (X \circ \iota^{\op}) \cong  \Lan_{\iota^{\op}} \left(\LocConst(-, X(*)) \circ \iota^{\op} \right) \cong \LocConst(-, X(*))
\]
where $\Lan_GH$ denotes the left Kan extension of $H$ along $G$ and $\iota : \FinSet \to \Profinite$ is the inclusion functor. Once this is established, we have proved Theorem~\ref{thm:main_colimit}.

The first and third isomorphisms 
\[
    X \cong \Lan_{\iota^{\op}} (X \circ \iota^{\op})
\] 
and 
\[
    \Lan_{\iota^{\op}} \left(\LocConst(-, X(*)) \circ \iota^{\op} \right) \cong \LocConst(-, X(*))
\]
come from Construction~\ref{const:1st3rd} below (the condensed set $\LocConst(-, X(*))$ satisfies the colimit condition as explained at the beginning of this proof, and $X$ does so by assumption). 

The middle isomorphism 
\[
    \Lan_{\iota^{\op}} (X \circ \iota^{\op}) \cong  \Lan_{\iota^{\op}} \left(\LocConst(-, X(*)) \circ \iota^{\op} \right)
\]
comes from Construction~\ref{const:middle} below, together with uniqueness of left Kan extensions up to isomorphism. 

\begin{construction}\label{const:middle} \href{https://github.com/leanprover-community/mathlib4/blob/ffb2f0dbf80ad18fa75f85bd0facee6737b59acf/Mathlib/Condensed/Discrete/Colimit.lean#L235-L242}{\faExternalLink} \href{https://github.com/leanprover-community/mathlib4/blob/ffb2f0dbf80ad18fa75f85bd0facee6737b59acf/Mathlib/Condensed/Discrete/Colimit.lean#L180-L186}{\faExternalLink}
    We construct an isomorphism $X \circ \iota^{\op} \cong \LocConst(-, X(*)) \circ \iota^{\op}$. Consider the functor $F : \FinSet \to \Set$ which takes a finite set $Y$ to the set $X(*)^Y$ of maps $Y \to X(*)$. We proceed by showing that both functors are isomorphic to this functor. The isomorphism $F \cong \LocConst(-, X(*)) \circ \iota^{\op}$ is trivial. The component of the isomorphism $X \circ \iota^{\op} \cong F$ at a finite set $Y$ is the following composition of isomorphisms
    \[
        X(Y) \cong X\left(\coprod_{y \in Y} *\right) \cong \prod_{y \in Y} X(*) \cong X(*)^Y.
    \]
    For the proof of naturality, we refer to the \Lean code.
\end{construction}

\begin{construction}\label{const:1st3rd} \href{https://github.com/leanprover-community/mathlib4/blob/ffb2f0dbf80ad18fa75f85bd0facee6737b59acf/Mathlib/Condensed/Discrete/Colimit.lean#L142-L145}{\faExternalLink}
    Let $F : \Profinite \to \Set$ be a presheaf satisfying the colimit condition (see Definition~\ref{def:colimit_condition}). We construct an isomorphism $F \cong \Lan_{\iota^{\op}} (F \circ \iota^{\op})$. 
    
    To explain the construction of this isomorphism more clearly, we introduce some notation. For a profinite set $S$, denote by $\dq(S)$ the poset of discrete quotients of $S$, and by $Q_S : \dq(S) \to \mathsf{FinSet}$ the functor $i \mapsto S_i$ defined in Construction~\ref{const:fintypeDiagram}. Let $S/\iota$ denote the over category consisting of morphisms $S \to \iota(X)$ in $\Profinite$ where $X$ runs over all finite sets. Denote by $P_S : S/\iota \to \mathsf{FinSet}$ the projection functor. Note that there is an obvious functor $\pi : \dq(S) \to S/\iota$ defined by mapping $i$ to $\pi_i : S \to S_i$. The triangle in the diagram
    \[\begin{tikzcd}
        {\dq(S)^{\op}} && {(S/\iota)^{\op}} \\
        & {\FinSet^{\op}} \\
        & \Set
        \arrow["\pi^{\op}", from=1-1, to=1-3]
        \arrow["{Q_S^{\op}}"', from=1-1, to=2-2]
        \arrow["{P_S^{\op}}", from=1-3, to=2-2]
        \arrow["{F\circ\iota^{\op}}", from=2-2, to=3-2]
    \end{tikzcd}\]
    is strictly commutative, and constructing the desired isomorphism amounts to constructing an isomorphism 
    \[
        \operatorname{colim} (F \circ \iota^{\op} \circ Q_S^{\op}) \cong \operatorname{colim} (F \circ \iota^{\op} \circ P_S^{\op})
    \]
    which is natural in the profinite set $S$. Omitting the proof of naturality, it suffices, by Proposition~\ref{prop:initial-final-limit}, to show that the functor $\pi$ is initial \href{https://github.com/leanprover-community/mathlib4/blob/ffb2f0dbf80ad18fa75f85bd0facee6737b59acf/Mathlib/Topology/Category/Profinite/Extend.lean#L85-L100}{\faExternalLink}. This amounts to showing that for every $S \to X \in S/\iota$, i.e.~for every continuous map from $S$ to a finite discrete set $X$, there exists a discrete quotient $i$ of $S$ with a map $S_i \to X$ such that the triangle
    \[\begin{tikzcd}
        & S \\
        {S_i} && X
        \arrow["{\pi_i}"', from=1-2, to=2-1]
        \arrow[from=1-2, to=2-3]
        \arrow[from=2-1, to=2-3]
    \end{tikzcd}\]
    commutes, and that for every discrete quotient $i$ of $S$ with two commutative triangles as above, there is a $j \leq i$ equalizing the two parallel maps in the diagram
    \[\begin{tikzcd}
        && S \\
        {S_j} & {S_i} && X
        \arrow["{\pi_j}"', from=1-3, to=2-1]
        \arrow["{\pi_i}", from=1-3, to=2-2]
        \arrow[from=1-3, to=2-4]
        \arrow[from=2-1, to=2-2]
        \arrow[shift left, from=2-2, to=2-4]
        \arrow[shift right, from=2-2, to=2-4]
    \end{tikzcd}\]
    The former condition follows from Lemma~\ref{lemma:factor_through} and the latter is trivial; one can always take $j = i$ because the projection map $\pi_i$ is surjective. 
\end{construction}

\subsection{Formalization}
To prove the colimit characterization of discrete condensed sets, we actually formalized more general constructions than those described informally in \S\ref{subsec:profinite_as_limit}-\S\ref{subsec:proof}. In what follows, we describe the formalization of this general setup. Throughout this section, three dots ``\lean{...}'' in a block of \Lean code signify an omitted proof.  

The functor $\pi$ described in Construction~\ref{const:1st3rd} is defined for any cone on a functor to the category of finite sets, and proved initial if the indexing category is cofiltered, the cone is limiting, and each projection map is epimorphic:
\begin{lstlisting}
variable {I : Type u} [SmallCategory I] {F : I ⥤ FintypeCat}
    (c : Cone <| F ⋙ toProfinite)

def functor : I ⥤ StructuredArrow c.pt toProfinite where
  obj i := StructuredArrow.mk (c.π.app i)
  map f := StructuredArrow.homMk (F.map f) (c.w f)

lemma functor_initial [IsCofiltered I] (hc : IsLimit c) [∀ i, Epi (c.π.app i)] : 
    Initial (functor c) := by
  ...
\end{lstlisting}
Then, given a presheaf $G$ on the category of profinite sets, valued in any category, a cocone is defined on the functor $F \circ \iota^{\op} \circ P_S^{\op}$ (here the notation is again borrowed from Construction~\ref{const:1st3rd}):
\begin{lstlisting}
variable {C : Type*} [Category C] (G : Profiniteᵒᵖ ⥤ C)

def cocone (S : Profinite) :
    Cocone (CostructuredArrow.proj toProfinite.op ⟨S⟩ ⋙ toProfinite.op ⋙ G) where
  pt := G.obj ⟨S⟩
  ι := {
    app := fun i ↦ G.map i.hom
    naturality := ... }

example : G.mapCocone c.op = (cocone G c.pt).whisker (functorOp c) := rfl
\end{lstlisting}
The \lean{example} below the definition is important: it shows that the cocone obtained by applying the presheaf \lean{G} to the cone \lean{c} is \emph{definitionally equal} to the cocone obtained by whiskering \lean{cocone G c.pt} with the functor $\pi^{\op}$ (in general, equality of cocones is not a desirable property to work with, but definitional equality is good). This means that under the same conditions that make the functor $\pi$ initial, together with an assumption that \lean{c} becomes a colimiting after applying \lean{G}, the cocone defined above is also colimiting:
\begin{lstlisting}
def isColimitCocone (hc : IsLimit c) [∀ i, Epi (c.π.app i)] 
    (hc' : IsColimit <| G.mapCocone c.op) : IsColimit (cocone G c.pt) := 
  (functorOp_final c hc).isColimitWhiskerEquiv _ _ hc'
\end{lstlisting}

We also provide the dual results for a covariant functor $G$ out of $\Profinite$. As an application of the dual (applied to the identity functor from $\Profinite$ to itself), we obtain a new way to write a profinite set as a cofiltered limit of finite discrete sets \href{https://github.com/leanprover-community/mathlib4/blob/ffb2f0dbf80ad18fa75f85bd0facee6737b59acf/Mathlib/Topology/Category/Profinite/Extend.lean#L192-L213}{\faExternalLink}.

\section{Discrete condensed $R$-modules}\label{sec:modules}
The goal of this section is to show that given any ring $R$, a condensed $R$-module is discrete if and only if its underlying condensed set is discrete. 

Consider the functor
\begin{align*}
    L : \Mod_R & \to \CondMod_R \\
    M & \mapsto \LocConst(-, M)
\end{align*}
from $R$-modules to condensed $R$-modules. We would like to make the analogous constructions to the ones in \S\ref{sec:functors}. A bad approach to this would be to try to repeat the whole story. With that approach, showing that $L$ is left adjoint to the forgetful functor $\CondMod_R \to \Mod_R$ would be significantly harder. One would have to prove that the counit is $R$-linear, and this would mean keeping track of the fibres of two locally constant maps and their addition. This is ill-suited for formalization, and the approach of reducing the problem for condensed modules to the case of underlying condensed sets is both easier, and provides the opportunity to formalize much more general results (e.g.~those in \S\ref{sec:constant}). 

We start by proving that the constant sheaf functor $\Mod_R \to \CondMod_R$ is fully faithful. Once that is established the result follows from Proposition~\ref{prop:isDiscrete_iff_forget} together with Proposition~\ref{prop:preservesSheafification_concrete}, because the underlying condensed set of a condensed $R$ module is given by postcomposing with the forgetful functor from $R$-modules to sets.

\begin{construction}\label{const:constantSheafModuleIso} \href{https://github.com/leanprover-community/mathlib4/blob/ffb2f0dbf80ad18fa75f85bd0facee6737b59acf/Mathlib/Condensed/Discrete/Module.lean#L119-L131}{\faExternalLink}
    We construct a natural isomorphism
    \[
        \underline{(-)} \cong L.
    \]
    For each $R$-module $M$, we need to give an isomorphism of condensed $R$-modules
    \[
        \underline{M} \cong \LocConst(-,M).
    \]
    This is given by the composition
    \[
        \underline{M} \cong \underline{\LocConst(*,M)} \cong \LocConst(-,M)
    \]
    where the first isomorphism is given by applying the constant sheaf functor to the obvious isomorphism of $R$-modules 
    \[
        M \cong \LocConst(*,M)
    \]
    and the second isomorphism is given by the counit of the constant sheaf adjunction. It remains to prove two things; that this counit is an isomorphism, and that the isomorphism constructed this way is natural in $M$. 
\end{construction}

\begin{lemma} \href{https://github.com/leanprover-community/mathlib4/blob/ffb2f0dbf80ad18fa75f85bd0facee6737b59acf/Mathlib/Condensed/Discrete/Module.lean#L111-L112}{\faExternalLink}
    Let $M$ be an $R$-module. The counit of the constant sheaf adjunction applied at the condensed $R$-module $\LocConst(-,M)$ is an isomorphism.
\end{lemma}
\begin{proof}
    Denote this map by 
    \[\epsilon_{L(M)} : \underline{\LocConst(*,M)} \to \LocConst(-,M) \]
    Since the forgetful functor $U : \Mod_R \to \Set$ reflects isomorphisms, the ``forgetful functor'' $F_U : \CondMod_R \to \CondSet$ (given as the sheaf-composition functor associated to $U$) also reflects isomorphisms by Proposition~\ref{prop:reflectsIso_etc}(3). It is therefore enough to check that $F_U(\epsilon_{L(M)})$ is an isomorphism. Using Lemma~\ref{lemma:constantSheafAdj_counit_w}, it suffices to show that $\epsilon'_{U\circ L(M)}$ is an isomorphism, where $\epsilon'$ denotes the counit of the constant sheaf adjunction for condensed sets. But this is the same as saying that the condensed set $\LocConst(-,M)$ is discrete, which we already know. 
\end{proof}

\begin{lemma}\label{lemma:constantSheafModuleIso_natural} \href{https://github.com/leanprover-community/mathlib4/blob/ffb2f0dbf80ad18fa75f85bd0facee6737b59acf/Mathlib/Condensed/Discrete/Module.lean#L119-L131}{\faExternalLink}
    The isomorphism in Construction~\ref{const:constantSheafModuleIso} is natural in the $R$-module $M$.
\end{lemma}
\begin{proof}
    This follows by unfolding the definitions and using some straightforward naturality arguments. We refer to the \Lean code for more details. 
\end{proof}

\begin{lemma}\label{lemma:LModuleFF} \href{https://github.com/leanprover-community/mathlib4/blob/ffb2f0dbf80ad18fa75f85bd0facee6737b59acf/Mathlib/Condensed/Discrete/Module.lean#L143-L145}{\faExternalLink}
    The functor $L : \Mod_R \to \CondMod_R$ is fully faithful.
\end{lemma}
\begin{proof}
    Using construction~\ref{const:constantSheafModuleIso}, we obtain an adjunction $L \dashv U$ where $U$ is the forgetful functor $\CondMod_R \to \Mod_R$. The obvious isomorphism of $R$-modules 
    \[
        M \cong \LocConst(*,M)
    \]
    is clearly natural in $M$, i.e.~it induces a natural isomorphism $\operatorname{Id} \cong U \circ L$. Now the result follows from Corollary~\ref{cor:adjunction-generality}.
\end{proof}

Lemma~\ref{lemma:LModuleFF} together with the general Proposition~\ref{prop:isDiscrete_iff_forget} concludes the proof of the main result (it has been established in \S\ref{sec:sheafification} and \S\ref{sec:functors} that Proposition~\ref{prop:isDiscrete_iff_forget} applies in this situation):

\begin{maintheorembody}\label{thm:main_module} \href{https://github.com/leanprover-community/mathlib4/blob/ffb2f0dbf80ad18fa75f85bd0facee6737b59acf/Mathlib/Condensed/Discrete/Characterization.lean#L112-L115}{\faExternalLink}
    Let $R$ be a ring. A condensed $R$-module is discrete if and only if its underlying condensed set is discrete. \qed
\end{maintheorembody}

\section{Conclusion}\label{sec:summary} 
We summarize the characterization of discrete condensed sets in Theorem~\ref{thm:summary}, and that of discrete condensed modules over a ring in Theorem~\ref{thm:summary_mod}. 

\begin{theorem}\label{thm:summary} \href{https://github.com/leanprover-community/mathlib4/blob/ffb2f0dbf80ad18fa75f85bd0facee6737b59acf/Mathlib/Condensed/Discrete/Characterization.lean#L74-L104}{\faExternalLink}
Let $L : \Set \to \CondSet$ denote the functor that takes a set $X$ to the sheaf of locally constant maps $\LocConst(-, X)$, and for a set $X$, denote by $\underline{X}$ the constant sheaf at $X$. Recall that these functors are both left adjoint to the underlying set functor $U : \CondSet \to \Set$, which takes a condensed set $X$ to its underlying set $X(*)$. The following conditions on a condensed set $X$ are equivalent. 
\begin{enumerate}
  \item $X$ is discrete, i.e.~there exists a set $Y$ and an isomorphism $X \cong \underline{Y}$. 
  \item $X$ is in the essential image of the functor $L$
  \item The component at $X$ of the counit of the adjunction $\underline{(-)} \dashv U$ is an isomorphism.
  \item The component at $X$ of the counit of the adjunction $L \dashv U$ is an isomorphism.
  \item For every profinite set $S = \varprojlim_i S_i$, the canonical map $\varinjlim_i X(S_i) \to X(S)$ is an isomorphism. 
\end{enumerate}
\end{theorem}
\begin{proof}
  The equivalence of (1)-(4) is established in \S\ref{sec:functors} where the isomorphism between the functors $L$ and $\underline{(-)}$ is constructed, and full faithfulness proved. For the equivalence with (5), see \S\ref{sec:colimit}.
\end{proof}

\begin{theorem}\label{thm:summary_mod} \href{https://github.com/leanprover-community/mathlib4/blob/ffb2f0dbf80ad18fa75f85bd0facee6737b59acf/Mathlib/Condensed/Discrete/Characterization.lean#L121-L158}{\faExternalLink}
  Let $R$ be a ring. Denote by $L : \Mod_R \to \CondMod_R$ the functor that takes an $R$-module $X$ to the sheaf of $R$-modules of locally constant maps $\LocConst(-, X)$, and for an $R$-module $X$, denote by $\underline{X}$ the constant sheaf at $X$. These functors are both left adjoint to the underlying module functor $U : \CondMod_R \to \Mod_R$, which takes a condensed $R$-module $X$ to its underlying module $X(*)$. The following conditions on a condensed $R$-module $X$ are equivalent. 
  \begin{enumerate}
    \item $X$ is discrete, i.e.~there exists an $R$-module $N$ and an isomorphism $X \cong \underline{N}$. 
    \item $X$ is in the essential image of the functor $L$
    \item The component at $X$ of the counit of the adjunction $\underline{(-)} \dashv U$ is an isomorphism.
    \item The component at $X$ of the counit of the adjunction $L \dashv U$ is an isomorphism.
    \item For every profinite set $S = \varprojlim_i S_i$, the canonical map $\varinjlim_i X(S_i) \to X(S)$ is an isomorphism. 
  \end{enumerate}
\end{theorem}
\begin{proof}
  The proof of the equivalence of (1)-(4) is identical to the condensed set case. For the equivalence with (5), we use the fact that a condensed module is discrete if and only if its underlying set is (Theorem~\ref{thm:main_module}), and that the forgetful functor from $R$-modules to sets preserves and reflects filtered colimits.
\end{proof}

\begin{remark}
  The variants of Theorems~\ref{thm:summary} and~\ref{thm:summary_mod} for light condensed sets and modules also hold \href{https://github.com/leanprover-community/mathlib4/blob/ffb2f0dbf80ad18fa75f85bd0facee6737b59acf/Mathlib/Condensed/Discrete/Characterization.lean#L193-L216}{\faExternalLink} \href{https://github.com/leanprover-community/mathlib4/blob/ffb2f0dbf80ad18fa75f85bd0facee6737b59acf/Mathlib/Condensed/Discrete/Characterization.lean#L230-L262}{\faExternalLink}. The statements are identical except that the last condition becomes 
  \begin{enumerate}
    \item[\textit{(5)}] \textit{For every light profinite set $S = \varprojlim_{n \in \N} S_n$, the canonical map $\varinjlim_n X(S_n) \to X(S)$ is an isomorphism.}
  \end{enumerate}
\end{remark}

\section*{Acknowledgements}
This work began as my first formalization project, depending on the definition of condensed objects in the Liquid Tensor Experiment (LTE). The work presented here is a much improved version of it, depending only on \mathlib. This improvement would not have been possible without the help of many people, to which I am very grateful and will try to list here. 

First, I would like to thank everyone involved in LTE for pioneering the formalization of condensed mathematics. 

I gave a talk about this work at the workshop on formalizing cohomology theories in Banff in 2023. I thank the organizers for the opportunity to give the talk, and more broadly all the participants for creating such a welcoming atmosphere at my first in-person encounter with the \Lean community. 

Shortly thereafter, I organized a masterclass in Copenhagen, where participants started the progress of formalizing condensed mathematics for \mathlib. I am grateful to everyone who took part in this, in particular to Adam Topaz and Kevin Buzzard for the lectures they delivered, and to Boris Kjær for co-organizing it with me. 

I wish to express my gratitude to the whole \mathlib community for their support and interest in this work, and in particular to Johan Commelin, Joël Riou, and Adam Topaz for thoroughly reviewing many pull requests related to the work presented here, which led to many improvements.

Finally, I would like to thank my advisor Dustin Clausen for explaining to me the importance of the colimit characterization of discrete condensed objects, for everything he has taught me about condensed mathematics more generally, and for supporting my decision to start focusing on  formalized mathematics.

I was supported by the Copenhagen Centre for Geometry and Topology through grant CPH-GEOTOP-DNRF151, which also provided financial support for the above-mentioned masterclass. 

\bibliographystyle{plainnat}
\bibliography{references}

\end{document}